\documentclass[a4paper,12pt,dvipdfmx]{amsart}

\usepackage{amsmath,amsthm,amssymb}
\usepackage[dvipdfmx]{graphicx}
\usepackage{here}
\usepackage{color}
\usepackage[all]{xy}
\xyoption{pdf}
\usepackage[top=30truemm,bottom=30truemm,left=25truemm,right=25truemm]{geometry}
\usepackage{ytableau}
\usepackage{hyperref}
\usepackage{tikz}

\newtheorem{thm}{Theorem}[subsection]
\newtheorem{lem}[thm]{Lemma}
\newtheorem{prop}[thm]{Proposition}
\newtheorem{cor}[thm]{Corollary}

\theoremstyle{definition}

\newtheorem{rem}[thm]{Remark}
\newtheorem{ex}[thm]{Example}

\numberwithin{equation}{section}

\author{Hideya Watanabe}
\address{(H. Watanabe) Osaka Central Advanced Mathematical Institute, Osaka Metropolitan University, Osaka, 558-8585, Japan}
\email{watanabehideya@gmail.com}
\date{\today}

\title{Stability of $\imath$canonical bases of irreducible finite type of real rank one}

\begin{document}
\maketitle
%\tableofcontents

\begin{abstract}
  It has been known since their birth in Bao and Wang's work that the $\imath$canonical bases of $\imath$quantum groups are not stable in general.
  In the author's previous work, the stability of $\imath$canonical bases of certain quasi-split types turned out to be closely related to the theory of $\imath$crystals.
  In this paper, we prove the stability of $\imath$canonical bases of irreducible finite type of real rank $1$, for which the theory of $\imath$crystals has not been developed, by means of global and local crystal bases.
\end{abstract}

\section{introduction}
Let $A = (a_{i,j})_{i,j \in I}$ be a symmetrizable generalized Cartan matrix, $\mathfrak{g}$ the associated Kac-Moody algebra, and $\mathbf{U} = U_q(\mathfrak{g})$ the Drinfeld and Jimbo's quantum group with weight lattice $X$.
Let $X^+$ denote the set of dominant weights.
For each $\lambda \in X^+$, let $V(\lambda)$ (resp., $V(-\lambda)$) denote the irreducible integrable highest (resp., lowest) weight $\mathbf{U}$-module of highest weight $\lambda$ (resp., lowest weight $-\lambda$) with highest weight vector $v_\lambda$ (resp., lowest weight vector $v_{-\lambda}$).
The canonical bases (also known as global crystal bases) of the negative part $\mathbf{U}^-$ and the positive part $\mathbf{U}^+$ of $\mathbf{U}$, and of $V(\pm \lambda)$ for all $\lambda \in X^+$, were constructed for type ADE in \cite{Lus90} and for general in \cite{Lus91} geometrically and in \cite{Kas91} algebraically.

In \cite{Lus92}, Lusztig constructed the canonical basis of the tensor product $V(-\lambda) \otimes V(\mu)$ for arbitrary $\lambda,\mu \in X^+$, from the canonical bases of $V(-\lambda)$ and $V(\mu)$.
A key ingredient of his construction is the quasi-$R$-matrix, which intertwines the bar-involutions on $\mathbf{U} \otimes \mathbf{U}$ and $\mathbf{U}$.

The canonical bases thus constructed are stable in the following sense.
For each $\lambda,\mu,\nu \in X^+$, there exists a unique $\mathbf{U}$-module homomorphism
$$
\pi_{\lambda,\mu,\nu}: V(-\lambda-\nu) \otimes V(\mu+\nu) \rightarrow V(-\lambda) \otimes V(\mu)
$$
which sends $v_{-\lambda-\nu} \otimes v_{\mu+\nu}$ to $v_{-\lambda} \otimes v_\mu$.
Then, each canonical basis element of $V(-\lambda-\nu) \otimes V(\mu+\nu)$ is sent to either a canonical basis element of $V(-\lambda) \otimes V(\mu)$ or $0$, and the kernel of $\pi_{\lambda,\mu,\nu}$ is spanned by a subset of the canonical basis.
In other words, the homomorphism $\pi_{\lambda,\mu,\nu}$ is a based $\mathbf{U}$-module homomorphism.

From this stability property, we see that for each $\zeta \in X$, the family
\begin{align}
  \{ V(-\lambda) \otimes V(\mu) \}_{\substack{\lambda,\mu \in X^+ \\ \mu-\lambda = \zeta}} \label{eq: proj system for Udot}
\end{align}
of $\mathbf{U}$-modules with $\pi_{\lambda,\mu,\nu}$ for each $\lambda,\mu,\nu \in X^+$ with $\mu-\lambda = \zeta$, forms a projective system of based $\mathbf{U}$-modules.
Then, the subspace $\dot{\mathbf{U}} 1_\zeta$ of the modified quantum group $\dot{\mathbf{U}} = \bigoplus_{\zeta \in X} \dot{\mathbf{U}} 1_\zeta$ with canonical basis can be regarded as the projective limit of the projective system above in a certain category of based $\mathbf{U}$-modules.
This construction led to an explicit description of the crystal basis of modified quantum group in \cite{Kas94}.

Our main interest in the present paper is the $\imath$quantum group counterpart of the construction above.
The $\imath$quantum group (also known as the quantum symmetric pair coideal subalgebra) $\mathbf{U}^\imath$ associated with an admissible pair $(I_\bullet, \tau)$ (in the sense of \cite[Definition 2.3]{Kol14}) and parameters $\varsigma_i \in \mathbb{Q}(q)^\times$ and $\kappa_i \in \mathbb{Q}(q)$ for $i \in I \setminus I_\bullet$ is a certain right coideal subalgebra of $\mathbf{U}$ which forms a quantum symmetric pair $(\mathbf{U}, \mathbf{U}^\imath)$.
For each Kac-Moody algebra $\mathfrak{g}$, there exists a quantum symmetric pair $(\mathbf{U}, \mathbf{U}^\imath) = (U_q(\mathfrak{g} \oplus \mathfrak{g}), U_q(\mathfrak{g}))$.
Such a quantum symmetric pair is said to be of diagonal type.
Therefore, the quantum group $\mathbf{U}$ itself is an instance of $\imath$quantum groups.

Let $w_\bullet$ denote the longest element of the Weyl group associated with $I_\bullet$ (by the definition of admissible pairs, $I_\bullet$ is of finite type).
Bao and Wang constructed a $\mathbf{U}^\imath$-module $V(w_\bullet\lambda, \mu)$ for each $\lambda,\mu \in X^+$ with a distinguished basis, called the $\imath$canonical basis, and $\mathbf{U}^\imath$-module homomorphisms
$$
\pi^\imath_{\lambda,\mu,\nu}: V(w_\bullet(\lambda+\tau\nu), \mu+\nu) \rightarrow V(w_\bullet\lambda,\mu)
$$
for finite type in \cite{BW18} and for the general case in \cite{BW21}.
To be more precise, the $\mathbf{U}^\imath$-module $V(w_\bullet\lambda, \mu)$ is obtained by restriction from the $\mathbf{U}$-submodule of $V(\lambda) \otimes V(\mu)$ generated by $v_{w_\bullet\lambda} \otimes v_\mu$, where $v_{w_\bullet\lambda} \in V(\lambda)$ denotes the unique canonical basis element of weight $w_\bullet\lambda$.
And, the $\mathbf{U}^\imath$-module homomorphism $\pi^\imath_{\lambda,\mu,\nu}$ is the unique one which sends $v_{w_\bullet(\lambda+\tau\nu)} \otimes v_{\mu+\nu}$ to $v_{w_\bullet\lambda} \otimes v_\mu$.
A key ingredient of their construction of the $\imath$canonical bases is the quasi-$K$-matrix, which intertwines the bar-involutions on $\mathbf{U}$ and $\mathbf{U}^\imath$.
The existence of quasi-$K$-matrix and bar-involution on $\mathbf{U}^\imath$ in general was formulated in \cite{BW18a} and proved in \cite{Kol21} after many partial results (see Section $1$ of {\it loc. cit.}).

Set $X^\imath := X/\{ \lambda+w_\bullet\tau\lambda \mid \lambda \in X \}$, and let $\overline{\cdot}: X \rightarrow X^\imath$ denote the quotient map.
Then, for each $\zeta \in X^\imath$, we obtain a projective system
\begin{align}
  \{ V(w_\bullet\lambda, \mu) \}_{\substack{\lambda,\mu \in X^+ \\ \overline{\mu+w_\bullet\lambda} = \zeta}} \label{eq: proj system for Uidot}
\end{align}
of $\mathbf{U}^\imath$-modules.
This projective system can be seen as a natural generalization of Lusztig's one \eqref{eq: proj system for Udot}.
In fact, they coincide with each other when the quantum symmetric pair is of diagonal type since the quasi-$K$-matrix and the $\mathbf{U}^\imath$-module $V(w_\bullet\lambda,\mu)$ become the quasi-$R$-matrix and $V(-\lambda) \otimes V(\mu)$, respectively.
However, in contrast to the projective system \eqref{eq: proj system for Udot}, the $\imath$canonical bases are not stable in the projective system \eqref{eq: proj system for Uidot}.
They are merely asymptotically stable; nevertheless this weaker stability property could still lead to the canonical basis (i.e., the $\imath$canonical basis) of the modified $\imath$quantum group in \cite{BW18}, \cite{BW21}.

On the other hand, in \cite{W21b}, the author proved that the $\imath$canonical bases are stable in the projective system \eqref{eq: proj system for Uidot} when $I_\bullet$ is empty, $a_{i,\tau(i)} \in \{2,0,-1\}$ for all $i \in I$, and when the parameters $\varsigma_i,\kappa_i$ are chosen appropriately.
As a result, he interprets the subspace $\dot{\mathbf{U}}^\imath 1_\zeta$ of the modified $\imath$quantum group $\dot{\mathbf{U}}^\imath = \bigoplus_{\zeta \in X^\imath} \dot{\mathbf{U}}^\imath 1_\zeta$ with $\imath$canonical basis as the projective limit of \eqref{eq: proj system for Uidot} in a certain category of based $\mathbf{U}^\imath$-modules.
In the proof, the theory of $\imath$crystals developed in \cite{W21a} plays a crucial role.

It is natural to expect that one can prove the stability of $\imath$canonical bases for general quantum symmetric pairs by developing the theory of $\imath$crystals.
The theory of $\imath$crystals in \cite{W21b} is based on many explicit calculation involving the quantum symmetric pairs of real rank $1$, just like the theory of crystals is based on calculation involving the quantum group of rank $1$.
Here, the real rank of a quantum symmetric pair refers the number of $\tau$-orbits in $I \setminus I_\bullet$.
The quantum group in a quantum symmetric pair of real rank $1$ considered in \cite{W21b} is either of type $A_1$, $A_1 \times A_1$, or $A_2$.
Hence, its structure is relatively simple.
In general, the quantum group in a quantum symmetric pair of real rank $1$ is not of finite type.
Even if we restrict our attention to a quantum symmetric pair of finite classical type, its rank can be arbitrarily high.
Hence, the same strategy as \cite{W21b} is not applicable for general quantum symmetric pairs.

In the present paper, we prove the stability of $\imath$canonical bases for the quantum symmetric pair of irreducible finite type of real rank $1$ without developing the theory of $\imath$crystals.
Here, ``irreducible'' means that the Dynkin diagram $I$, extended by adding edges between $i$ and $\tau(i)$ for all $i \in I$, is connected as a graph.
This is the first step toward the generalization of the stability theorem of $\imath$canonical bases to general quantum symmetric pairs.
Since the stability of $\imath$canonical bases is closely related to $\imath$crystals, the author expects that our new stability theorem, in turn, stimulates an attempt to extend the theory of $\imath$crystals to general quantum symmetric pairs.

Let us summarize our proof of the stability of $\imath$canonical bases.
First, we study the $\mathbf{U}$-module structure of $V(w_\bullet\lambda,\mu)$ by investigating the crystal structure of its crystal basis.
This enables us to construct a $\mathbf{U}$-module homomorphism
\begin{align}
  V(w_\bullet(\lambda+\tau\nu), \mu+\nu) \rightarrow V(\nu+w_\bullet\tau\nu) \otimes V(w_\bullet\lambda, \mu). \label{eq: key based U-hom}
\end{align}
Note that $\nu+w_\bullet\tau\nu$ is dominant.
Next, verifying a sufficient condition for a $\mathbf{U}$-module homomorphism to be based, which is given in Proposition \ref{prop: suff cond to be based hom}, we prove that the $\mathbf{U}$-module homomorphism above is based.
Finally, we prove that there exists a based $\mathbf{U}^\imath$-module homomorphism $V(\nu+w_\bullet\tau\nu) \rightarrow \mathbb{Q}(q)$ (here, $\mathbb{Q}(q)$ is the trivial $\mathbf{U}^\imath$-module with $\imath$canonical basis $\{1\}$) which sends the highest weight vector to $1$.
Composing this based $\mathbf{U}^\imath$-module homomorphism with the based $\mathbf{U}$-module homomorphism in the second step, we obtain a based $\mathbf{U}^\imath$-module homomorphism, which is identical to $\pi^\imath_{\lambda,\mu,\nu}$.
The first and second steps are applicable for general quantum symmetric pairs.
The final step is achieved by studying the quantum symmetric pairs of irreducible finite type of real rank $1$ one-by-one; there are only eight kinds of such quantum symmetric pairs.

The paper is organized as follows.
In Section \ref{sec: qua grp}, we review well-known results concerning canonical, or global crystal, and crystal bases, and based $\mathbf{U}$-modules.
Then, we give a sufficient condition for two based $\mathbf{U}$-modules with a crystal morphism between their crystal bases having a based $\mathbf{U}$-module homomorphism which lifts the crystal morphism in Proposition \ref{prop: suff cond to be based hom}.
In Section \ref{sec: iqua grp}, after recalling the definition of $\imath$quantum groups and results of \cite{BW21} about based $\mathbf{U}^\imath$-modules, we construct the based $\mathbf{U}$-module homomorphism \eqref{eq: key based U-hom}.
In Section \ref{sec: main results}, we finish our proof of the stability of $\imath$canonical bases by studying certain $\mathbf{U}^\imath$-modules for each quantum symmetric pair of irreducible finite type of real rank $1$ one-by-one.

\subsection*{Acknowledgement}
This work was supported by JSPS KAKENHI Grant Numbers JP20K14286 and JP21J00013.

\section{Quantum Group}\label{sec: qua grp}
The purpose of this section is to fix our notation about the quantum groups, and then prove Proposition \ref{prop: suff cond to be based hom}, which provides a sufficient condition for a crystal morphism between the crystal bases of two based $\mathbf{U}$-modules to be lifted to a based $\mathbf{U}$-module homomorphism.

\subsection{Quantum group}\label{subsec: qua grp}
Let $A = (a_{i,j})_{i,j \in I}$ be a symmetrizable generalized Cartan matrix with a symmetrizing matrix $D = \operatorname{diag}(d_i \mid i \in I)$ with $d_i \in \mathbb{Z}_{> 0}$ being relatively prime.
We often identify $I$ with the Dynkin diagram of $A$.
Let $Y$ and $X$ be finitely generated free abelian groups with a perfect pairing $\langle,\rangle: Y \times X \rightarrow \mathbb{Z}$.
Let $\Pi^\vee = \{ h_i \mid i \in I \} \subset Y$ and $\Pi = \{ \alpha_i \mid i \in I \} \subset X$ be linearly independent subsets satisfying
$$
\langle h_i, \alpha_j \rangle = a_{i,j}
$$
for all $i,j \in I$.
Let $W$ denote the Weyl group associated with the generalized Cartan matrix $A$.
For each $i \in I$, let $s_i \in W$ denote the simple reflection.

Let $q$ be an indeterminate.
For each $i \in I$, $n \in \mathbb{Z}$, and $m \in \mathbb{Z}_{\geq 0}$, set
$$
q_i := q^{d_i},\ [n]_i := \frac{q_i^n-q_i^{-n}}{q_i-q_i^{-1}},\ [m]_i! := \prod_{n=1}^m [n]_i.
$$
When $d_i = 1$, we sometimes omit the subscript ``$i$'' from notation above.

Let $\mathbf{U}$ denote the quantum group.
Namely, $\mathbf{U}$ is the unital associative $\mathbb{Q}(q)$-algebra with generators $\{ E_i,F_i,K_h \mid i \in I,\ h \in Y \}$ subject to the following relations: For each $i,j \in I$ and $h,h_1,h_2 \in Y$,
\begin{align*}
  &K_0 = 1,\ K_{h_1} K_{h_2} = K_{h_1 + h_2}, \\
  &K_h E_i = q^{\langle h,\alpha_i \rangle} E_i K_h,\ K_h F_i = q^{\langle h, -\alpha_i \rangle} F_i K_h, \\
  &E_i F_j - F_j E_i = \delta_{i,j} \frac{K_i - K_i^{-1}}{q_i-q_i^{-1}}, \\
  &S_{i,j}(E_i, E_j) = S_{i,j}(F_i, F_j) = 0\ \text{ if } i \neq j,
\end{align*}
where
$$
K_i := K_{d_i h_i},\ S_{i,j}(x,y) := \sum_{r+s = 1-a_{i,j}} (-1)^s \frac{1}{[r]_i![s]_i!} x^r y x^s.
$$
The $\mathbf{U}$ is a Hopf algebra with comultiplication map $\Delta$ given by
\begin{align*}
  &\Delta(E_i) = E_i \otimes 1 + K_i \otimes E_i, \\
  &\Delta(F_i) = 1 \otimes F_i + F_i \otimes K_i^{-1}, \\
  &\Delta(K_h) = K_h \otimes K_h
\end{align*}
for all $i \in I$, $h \in Y$.

Let $\overline{\cdot}$ denote the bar-involution on $\mathbf{U}$, i.e., the $\mathbb{Q}$-algebra automorphism such that
$$
\overline{E_i} = E_i, \ \overline{F_i} = F_i,\ \overline{K_h} = K_{-h},\ \bar{q} = q^{-1}
$$
for all $i \in I$, $h \in Y$.

Let $\rho$ denote the anti-algebra automorphism on $\mathbf{U}$ such that
\begin{align}
  \rho(E_i) = q_i^{-1} F_iK_i,\ \rho(F_i) = q_iK_i^{-1} E_i,\ \rho(K_h) = K_h \label{eq: def of rho}
\end{align}
for all $i \in I$, $h \in Y$.

For each $i \in I$, let $T_i$ denote both the algebra automorphism $T''_{i,1}$ on $\mathbf{U}$ in \cite[Proposition 37.1.2]{Lus10} and the automorphism $T''_{i,1}$ on integrable $\mathbf{U}$-modules in \cite[5.2.1]{Lus10}.
For each $w \in W$ with a reduced expression $w = s_{i_1} \cdots s_{i_r}$, set $T_w := T_{i_1} \cdots T_{i_r}$.

Let $\mathbf{U}^+$ (resp., $\mathbf{U}^-$) denote the subalgebra of $\mathbf{U}$ generated by $\{ E_i \mid i \in I \}$ (resp., $\{ F_i \mid i \in I \}$).
Let $(\mathcal{L}(\pm\infty), \mathcal{B}(\pm\infty))$ and $\mathbf{B}(\pm\infty)$ denote the crystal base and global crystal basis of $\mathbf{U}^\mp$ in \cite[Theorems 4 and 6]{Kas91}.
In this paper, crystal lattices are considered over the subring $\mathbf{A}_\infty$ of $\mathbb{Q}(q)$ consisting of all rational functions regular at $q = \infty$.
Let $G_{\pm\infty}: \mathcal{B}(\pm\infty) \rightarrow \mathbf{B}(\pm\infty)$ denote the bijection such that $G_{\pm\infty}(b) + q^{-1}\mathcal{L}(\pm\infty) = b$ for all $b \in \mathcal{B}(\pm\infty)$.
Set $b_{\pm\infty} := 1 + q^{-1} \mathcal{L}(\pm\infty) \in \mathcal{B}(\pm\infty)$.

%Let $\omega$ denote the Chevalley involution on $\mathbf{U}$, that is, an involutive algebra automorphism satisfying
%\begin{align}
%  \omega(E_i) = F_i,\ \omega(K_h) = K_{-h} \label{eq: def of Chevalley inv}
%\end{align}
%for all $i \in I$, $h \in Y$.
%This restricts to an algebra isomorphism $\mathbf{U}^- \rightarrow \mathbf{U}^+$, and then induces a bijection
%$$
%\omega: \mathcal{B}(\infty) \rightarrow \mathcal{B}(-\infty);\ b \mapsto \omega(G_\infty(b)) + q^{-1} \mathcal{L}(-\infty).
%$$
%The inverse of this bijection is also denoted by $\omega$.

\subsection{Crystal}\label{subsec: crystal}
Let $\mathcal{B}$ be a crystal in the sense of \cite[Section 1.2]{Kas93}.
For each $\lambda \in X$, set
$$
\mathcal{B}_\lambda := \{ b \in \mathcal{B} \mid \operatorname{wt}(b) = \lambda \}.
$$
We say that $b \in \mathcal{B}_\lambda$ is a highest weight element of weight $\lambda$ if $\tilde{E}_i b = 0$ for all $i \in I$.
Let $\mathcal{B}^{\text{hi}} \subset \mathcal{B}$ denote the set of highest weight elements.

For each $b \in \mathcal{B}$, let $C(b)$ denote the connected component of $\mathcal{B}$ containing $b$.

%Let $\mathcal{B}^\vee := \{ b^\vee \mid b \in \mathcal{B} \}$ denote the crystal obtained from $\mathcal{B}$ in the same way as \cite[Section 1.2]{Kas93}.
%For example, the assignment
%$$
%\mathcal{B}(-\infty) \rightarrow \mathcal{B}(\infty)^\vee;\ b \mapsto \omega(b)^\vee
%$$
%gives rise to a crystal isomorphism.

Given two crystals $\mathcal{B}_1, \mathcal{B}_2$, their tensor product $\mathcal{B}_1 \otimes \mathcal{B}_2$ is the crystal with the following structure:
\begin{align}
\begin{split}
  &\operatorname{wt}(b_1 \otimes b_2) = \operatorname{wt}(b_1) + \operatorname{wt}(b_2), \\
  &\varepsilon_i(b_1 \otimes b_2) = \max(\varepsilon_i(b_1) - \langle h_i,\operatorname{wt}(b_2) \rangle, \varepsilon_i(b_2)), \\
  &\varphi_i(b_1 \otimes b_2) = \max(\varphi_i(b_1), \varphi_i(b_2) + \langle h_i, \operatorname{wt}(b_1) \rangle), \\
  &\tilde{E}_i(b_1 \otimes b_2) = \begin{cases}
    \tilde{E}_i b_1 \otimes b_2 & \text{ if } \varepsilon_i(b_1) > \varphi_i(b_2), \\
    b_1 \otimes \tilde{E}_i b_2 & \text{ if } \varepsilon_i(b_1) \leq \varphi_i(b_2),
  \end{cases} \\
  &\tilde{F}_i(b_1 \otimes b_2) = \begin{cases}
    b_1 \otimes \tilde{F}_i b_2 & \text{ if } \varepsilon_i(b_1) < \varphi_i(b_2), \\
    \tilde{F}_i b_1 \otimes b_2 & \text{ if } \varepsilon_i(b_1) \geq \varphi_i(b_2).
  \end{cases}
\end{split} \label{eq: tensor crystal}
\end{align}
Note that this structure is different from that in \cite{Kas93} due to difference of convention.

The following lemma is easily deduced from the above.

\begin{lem}\label{lem: hw element in tensor crystal}
Let $\mathcal{B}_1,\mathcal{B}_2$ be crystals such that
$$
\varepsilon_i(b) = \max\{ k \in \mathbb{Z}_{\geq 0} \mid \tilde{E}_i^k b \neq 0 \},\ \varphi_i(b) = \max\{ k \in \mathbb{Z}_{\geq 0} \mid \tilde{F}_i^k b \neq 0 \}
$$
for all $i \in I$, $b \in \mathcal{B}_1, \mathcal{B}_2$.
Let $(b_1, b_2) \in \mathcal{B}_1 \times \mathcal{B}_2$.
Then, we have $b_1 \otimes b_2 \in (\mathcal{B}_1 \otimes \mathcal{B}_2)^{\text{hi}}$ if and only if $b_2 \in \mathcal{B}_2^{\text{hi}}$ and $\varepsilon_i(b_1) \leq \langle h_i, \operatorname{wt}(b_2) \rangle$ for all $i \in I$.
\end{lem}

\subsection{Irreducible module $V(\pm \lambda)$}\label{subsec: V(pm lm)}
Let $X^+$ denote the set of dominant weights:
$$
X^+ = \{ \lambda \in X \mid \langle h_i,\lambda \rangle \geq 0\ \text{ for all } i \in I \}.
$$
For each $\lambda \in X^+$, let $V(\lambda)$ (resp., $V(-\lambda)$) denote the irreducible integrable highest (resp., lowest) weight $\mathbf{U}$-module of highest weight $\lambda$ (resp., lowest weight $-\lambda$).
Let $(\mathcal{L}(\pm\lambda), \mathcal{B}(\pm\lambda))$ and $\mathbf{B}(\pm\lambda)$ denote the crystal base and global crystal basis of $V(\pm\lambda)$ \cite[Theorems 2 and 6]{Kas91}.
For $w \in W$, let $b_{\pm w\lambda} \in \mathcal{B}(\pm\lambda)$ and $v_{\pm w\lambda} \in \mathbf{B}(\pm\lambda)$ denote the unique elements of weight $\pm w \lambda$.

\begin{rem}\normalfont\label{rem: finite type}
  Suppose that the Dynkin diagram $I$ is of finite type.
  Let $w_0 \in W$ denote the longest element.
  Then, for each $\lambda \in X^+$ and $w \in W$, the symbol $v_{-w\lambda}$ represents both the vectors $v_{-w\lambda} \in V(-\lambda)$ and $v_{ww_0(-w_0\lambda)} \in V(-w_0 \lambda)$.
  In order to make our notation consistent, we identify $V(-\lambda)$ with $V(-w_0 \lambda)$ under the $\mathbf{U}$-module isomorphism $V(-\lambda) \rightarrow V(-w_0 \lambda)$ which sends $v_{-\lambda}$ to $v_{w_0(-w_0\lambda)}$.
\end{rem}

Let $(,)$ denote the inner product on $V(\lambda)$ such that $(v_\lambda,v_\lambda) = 1$ and
$$
(xu,v) = (u,\rho(x)v)
$$
for all $x \in \mathbf{U}$, $u,v \in V(\lambda)$, where $\rho$ is the anti-algebra automorphism on $\mathbf{U}$ given in \eqref{eq: def of rho}.
By \cite[Lemma 19.1.4]{Lus10}, we have
$$
(G(b_1), G(b_2)) \in \delta_{b_1,b_2} + q^{-1} \mathbf{A}_\infty
$$
for all $b_1,b_2 \in \mathcal{B}(\lambda)$.

%Given a $\mathbf{U}$-module $M$, let $M^\omega$ denote the $\mathbf{U}$-module with a $\mathbb{Q}(q)$-linear isomorphism
%$$
%M \rightarrow M^\omega;\ m \mapsto m^\omega
%$$
%and the $\mathbf{U}$-module structure
%$$
%x m^\omega = (\omega(x) m)^\omega
%$$
%for all $x \in \mathbf{U}$, $m \in M$, where $\omega$ denote the Chevalley involution on $\mathbf{U}$ given in \eqref{eq: def of Chevalley inv}.
%We identify $V(-\lambda)$ with $V(\lambda)^\omega$ via the $\mathbf{U}$-module isomorphism sending $v_{-\lambda}$ to $v_\lambda^\omega$.

For each $\lambda \in X^+$, let
$$
\pi_{\pm \lambda}: \mathbf{U}^\mp \rightarrow V(\pm \lambda)
$$
denote the $\mathbf{U}^\mp$-module homomorphism given by $\pi_{\pm \lambda}(x) = xv_{\pm \lambda}$.
By \cite[Theorem 5]{Kas91}, it induces a map $\pi_{\pm \lambda}: \mathcal{B}(\pm \infty) \rightarrow \mathcal{B}(\pm \lambda) \sqcup \{0\}$, and the induced map restricts to a bijection
$$
\mathcal{B}(\pm \infty; \pm\lambda) := \{ b \in \mathcal{B}(\pm \infty) \mid \pi_{\pm \lambda}(b) \neq 0 \} \rightarrow \mathcal{B}(\pm \lambda).
$$
Let
\begin{align}
  \iota_{\pm \lambda}: \mathcal{B}(\pm \lambda) \rightarrow \mathcal{B}(\pm \infty; \pm \lambda) \label{eq: iota_pm lm}
\end{align}
denote its inverse.
For each $b \in \mathcal{B}(\pm\lambda)$, we have the following:
\begin{align*}
  &G(b) = G_{\pm\infty}(\iota_{\pm\lambda}(b)) v_{\pm\lambda}, \\
  &\iota_{\pm\lambda}(b_{\pm\lambda}) = b_{\pm\infty}, \\
  &\operatorname{wt}(\iota_{\pm\lambda}(b)) = \operatorname{wt}(b)\mp\lambda, \\
  &\varepsilon_i(\iota_\lambda(b)) = \varepsilon_i(b),\ \varphi_i(\iota_{-\lambda}(b)) = \varphi_i(b), \\
  &\tilde{F}_i(\iota_\lambda(b)) = \iota_\lambda(\tilde{F}_i b) \text{ if } \tilde{F}_i b \neq 0,\ \tilde{E}_i \iota_{-\lambda}(b) = \iota_{-\lambda}(\tilde{E}_i b) \text{ if } \tilde{E}_i b \neq 0
\end{align*}
for all $i \in I$, $b \in \mathcal{B}(\pm \lambda)$.

\subsection{Based module}\label{subsec: based U-mod}
Set $\mathcal{A} := \mathbb{Q}[q,q^{-1}]$.
Let $\dot{\mathbf{U}} = \bigoplus_{\lambda,\mu \in X} 1_\lambda \dot{\mathbf{U}} 1_\mu$ denote the modified quantum group, and ${}_{\mathcal{A}} \dot{\mathbf{U}}$ its $\mathcal{A}$-form.

Following \cite[Section 2.1]{BW16}, we define a based $\mathbf{U}$-module to be an integrable $\mathbf{U}$-module $M$ equipped with a linear basis $\mathbf{B}_M$ satisfying the following:
\begin{itemize}
  \item $\mathbf{B}_M \cap M_\lambda$ is a basis of $M_\lambda$ for all $\lambda \in X$, where
  $$
  M_\lambda := \{ m \in M \mid K_h m = q^{\langle h,\lambda \rangle}m \text{ for all } h \in Y \}.
  $$
  \item ${}_{\mathcal{A}}M := \mathcal{A} \mathbf{B}_M$ is a ${}_{\mathcal{A}} \dot{\mathbf{U}}$-submodule. We call it the $\mathcal{A}$-form of $M$.
  \item The $\mathbb{Q}$-linear map $\overline{\cdot}: M \rightarrow M$ sending $q^n v$ to $q^{-n} v$ for all $n \in \mathbb{Z}$ and $v \in \mathbf{B}_M$ satisfies $\overline{xm} = \bar{x} \bar{m}$ for all $x \in \mathbf{U}$, $m \in M$. We call it the bar-involution on $M$.
  \item Setting $\mathcal{L}_M := \mathbf{A}_\infty \mathbf{B}_M$ and $\mathcal{B}_M := \{ v + q^{-1}\mathcal{L}_M \mid v \in \mathbf{B}_M \}$, the pair $(\mathcal{L}_M, \mathcal{B}_M)$ forms a crystal base of $M$.
\end{itemize}

Let $(M,\mathbf{B}_M)$ be a based $\mathbf{U}$-module.
Then, the quotient map $\mathrm{ev}_\infty: \mathcal{L}_M \rightarrow \mathcal{L}_M/q^{-1}\mathcal{L}_M$ restricts to a $\mathbb{Q}$-linear isomorphism
$$
\mathcal{L}_M \cap {}_{\mathcal{A}} M \cap \overline{\mathcal{L}_M} \rightarrow \mathcal{L}_M/q^{-1}\mathcal{L}_M.
$$
Let $G_M$ denote its inverse.
We sometimes omit the subscript $M$ of $G_M$ for simplicity.

A based submodule of a based module $(M, \mathbf{B}_M)$ is a $\mathbf{U}$-submodule $N \subset M$ spanned by a subset $\mathbf{B}_N \subset \mathbf{B}_M$.
Note that $(N, \mathbf{B}_N)$ is a based $\mathbf{U}$-module in its own right.

Let $(N, \mathbf{B}_N)$ be a based submodule of a based $\mathbf{U}$-module $(M, \mathbf{B}_M)$.
Then,
$$
(M/N, \{ v + N \mid v \in \mathbf{B}_M \setminus \mathbf{B}_N \})
$$
is a based $\mathbf{U}$-module.

Let $(M, \mathbf{B}_M), (N, \mathbf{B}_N)$ be based $\mathbf{U}$-modules, and $f: M \rightarrow N$ a $\mathbf{U}$-module homomorphism.
We say that $f$ is a based $\mathbf{U}$-module homomorphism if $f(\mathbf{B}_M) \subset \mathbf{B}_N \sqcup \{0\}$ and $\operatorname{Ker} f$ is a based submodule.
This definition is reformulated as follows.

\begin{lem}\label{lem: criterion of based hom}
  Let $(M, \mathbf{B}_M), (N, \mathbf{B}_N)$ be based $\mathbf{U}$-modules, and $f: M \rightarrow N$ a $\mathbf{U}$-module homomorphism.
  Then, $f$ is a based $\mathbf{U}$-module homomorphism if and only if it satisfies the following:
  \begin{itemize}
    \item $f(\mathcal{L}_M) \subset \mathcal{L}_N$; it induces a map
    $$
    \phi: \mathcal{B}_M \rightarrow \mathcal{B}_N;\ b \mapsto \mathrm{ev}_\infty(f(G_M(b))).
    $$
    \item $f({}_{\mathcal{A}}M) \subset {}_{\mathcal{A}}N$.
    \item $f \circ \psi_M = \psi_N \circ f$, where $\psi_M,\psi_N$ denote the bar-involution on $M,N$, respectively.
    \item $\phi$ is injective on $\{ b \in \mathcal{B}_M \mid \phi(b) \neq 0 \}$.
  \end{itemize}
\end{lem}

\begin{proof}
  The ``only if'' part is obvious.
  Let us prove the opposite direction.
  By the assumption on $f$, we see that
  $$
  f(G_M(b)) \in \mathcal{L}_N \cap {}_{\mathcal{A}}N \cap \overline{\mathcal{L}_N}
  $$
  for all $b \in \mathcal{B}_M$.
  Hence, we obtain
  \begin{align}
    f(G_M(b)) = (G_N \circ \mathrm{ev}_\infty)(f(G_M(b))) = G_N(\phi(b)) \label{eq: image of G_M(b)}
  \end{align}
  for all $b \in \mathcal{B}_M$.
  This implies that $f(\mathbf{B}_M) \subset \mathbf{B}_N \sqcup \{0\}$.
  In order to complete the proof, let us investigate $\operatorname{Ker} f$.
  Let $v \in \operatorname{Ker} f$, and write
  $$
  v = \sum_{b \in \mathcal{B}_M} c_b G_M(b)
  $$
  for some $c_b \in \mathbb{Q}(q)$.
  Then, by equation \eqref{eq: image of G_M(b)}, we have
  $$
  0 = f(v) = \sum_{\substack{b \in \mathcal{B}_M \\ \phi(b) \neq 0}} c_b G_N(b).
  $$
  Since $\phi$ is injective on $\{ b \in \mathcal{B}_M \mid \phi(b) \neq 0 \}$, we must have $c_b = 0$ for all $b \in \mathcal{B}_M$ with $\phi(b) \neq 0$.
  Therefore, we obtain
  $$
  v = \sum_{\substack{b \in \mathcal{B}_M \\ \phi(b) = 0}} c_b G_M(b).
  $$
  Consequently,
  $$
  \operatorname{Ker} f = \mathbb{Q}(q)\{ G_M(b) \mid b \in \mathcal{B}_M \text{ and } \phi(b) = 0 \}.
  $$
  This implies that $f$ is based.
  Thus, the proof completes.
\end{proof}

\begin{ex}\normalfont\label{ex: based mods}
  Let $\lambda,\mu,\lambda_1,\ldots,\lambda_r \in X^+$.
  \begin{enumerate}
    \item $(V(\pm\lambda), \mathbf{B}(\pm\lambda))$ is a based $\mathbf{U}$-module.
    \item\label{ex: based mod V(lm1,...,lmr)} Let $\mathbf{B}(\lambda_1,\ldots,\lambda_r)$ denote the canonical basis of $V(\lambda_1,\ldots,\lambda_r) := V(\lambda_1) \otimes \cdots \otimes V(\lambda_r)$ constructed in \cite[Theorem 2.9]{BW16}.
    Then, $(V(\lambda_1,\ldots,\lambda_r), \mathbf{B}(\lambda_1,\ldots,\lambda_r))$ is a based $\mathbf{U}$-module.
    Its crystal basis $\mathcal{B}(\lambda_1,\ldots,\lambda_r)$ is the tensor product $\mathcal{B}(\lambda_1) \otimes \cdots \otimes \mathcal{B}(\lambda_r)$.
    \item\label{ex: based mods chi} By \cite[Proposition 25.1.2]{Lus10}, for each $\lambda,\mu \in X^+$, there exists a unique based $\mathbf{U}$-module homomorphism
    $$
    \chi_{\lambda,\mu}: V(\lambda+\mu) \rightarrow V(\lambda, \mu)
    $$
    such that $\chi_{\lambda,\mu}(v_{\lambda+\mu}) = v_\lambda \otimes v_\mu$.
    \item\label{ex: based mods delta} By \cite[Proposition 25.1.4]{Lus10}, for each $\lambda \in X^+$, there exists a unique based $\mathbf{U}$-module homomorphism
    $$
    \delta_\lambda: V(-\lambda) \otimes V(\lambda) \rightarrow \mathbb{Q}(q)
    $$
    such that $\delta_\lambda(v_{-\lambda} \otimes v_\lambda) = 1$.
    Here, we identify $(\mathbb{Q}(q), \{1\})$ with the based $\mathbf{U}$-module $(V(0), \mathbf{B}(0))$.
  \end{enumerate}
\end{ex}

Let $\mathcal{C}^+$ denote the category of based $\mathbf{U}$-modules and based homomorphisms consisting of $(M, \mathbf{B}_M)$ with finite-dimensional weight spaces satisfying the following: There exist finitely many $\lambda_1,\ldots,\lambda_r \in X$ such that for each $\lambda \in X$, we have $M_\lambda \neq 0$ only if $\lambda \in \bigcup_{k=1}^r \{ \lambda_k - \alpha \mid \alpha \in \sum_{i \in I} \mathbb{Z}_{\geq 0} \alpha_i \}$.
In particular, $M$ is semisimple with simple components of the form $V(\lambda)$, $\lambda \in X^+$.

Let $(M, \mathbf{B}_M) \in \mathcal{C}^+$.
For each $\lambda \in X^+$, let $I_\lambda(M)$ denote the sum of all submodules of $M$ isomorphic to $V(\lambda)$.
Set
\begin{align*}
  &M[> \lambda] := \bigoplus_{\substack{\mu \in X^+ \\ \mu > \lambda}} I_\mu(M),\quad M[\geq \lambda] := M[> \lambda] \oplus I_\lambda(M), \\
  &\mathcal{B}_M[> \lambda] := \bigsqcup_{\substack{b' \in \mathcal{B}_M^{\text{hi}} \\ \operatorname{wt}(b') > \lambda}} C(b'),\quad \mathcal{B}_M[\geq \lambda] := \bigsqcup_{\substack{b' \in \mathcal{B}_M^{\text{hi}} \\ \operatorname{wt}(b') \geq \lambda}} C(b'),
\end{align*}
where $<$ denotes the partial order on $X^+$ defined by saying $\lambda \leq \mu$ if and only if $\mu - \lambda \in \sum_{i \in I} \mathbb{Z}_{\geq 0} \alpha_i$.
Recall that $C(b')$ denotes the connected component of $\mathcal{B}_M$ containing $b'$ (cf. Subsection \ref{subsec: crystal}).
By the same way as the proofs of \cite[Propositions 27.1.7 and 27.1.8]{Lus10}, we see that the submodules $M[> \lambda]$ and $M[\geq \lambda]$ are based submodules with crystal basis $\mathcal{B}_M[> \lambda]$ and $\mathcal{B}_M[\geq \lambda]$, respectively, and that there exists a based $\mathbf{U}$-module isomorphism $M[\geq \lambda]/M[> \lambda] \rightarrow V(\lambda)^{|\mathcal{B}_{M,\lambda}^{\text{hi}}|}$ sending $G_M(b)$ to $v_\lambda^b$ for all $b \in \mathcal{B}_{M, \lambda}^{\text{hi}}$, where $v_\lambda^b$ denotes the highest weight vector of the $b$-th component.

For each $\lambda \in X^+$, let
$$
p_\lambda: M = \bigoplus_{\mu \in X^+} I_\mu(M) \rightarrow I_\lambda(M)
$$
denote the projection.
For each $b \in \mathcal{B}_M^{\text{hi}}$, set
$$
v_b := p_{\operatorname{wt}(b)}(G_M(b)) \in I_{\operatorname{wt}(b)}(M).
$$
By the above, we see that $v_b$ is a highest weight vector of weight $\operatorname{wt}(b)$ and that
$$
v_b \in \mathcal{L}_M,\ \mathrm{ev}_\infty(v_b) = b.
$$

For each $b \in \mathcal{B}_M^{\text{hi}}$, set
$$
M[\geq b] := M[> \operatorname{wt}(b)] \oplus \mathbf{U} v_b,\quad \mathcal{B}_M[\geq b] := \mathcal{B}_M[> \operatorname{wt}(b)] \sqcup C(b).
$$
The connected component $C(b)$ is isomorphic to $\mathcal{B}(\operatorname{wt}(b))$.
Let
$$
\iota_b: C(b) \rightarrow \mathcal{B}(\operatorname{wt}(b)) \xrightarrow{\iota_{\operatorname{wt}(b)}} \mathcal{B}(\infty; \operatorname{wt}(b))
$$
denote the composition of the isomorphism $C(b) \rightarrow \mathcal{B}(\operatorname{wt}(b))$ and $\iota_{\operatorname{wt}(b)}$ (cf. equation \eqref{eq: iota_pm lm}).

\begin{lem}\label{lem: filtration on based mod}
  Let $(M, \mathbf{B}_M) \in \mathcal{C}^+$, and $b \in \mathcal{B}_M^{\text{hi}}$.
  Then, $M[\geq b]$ is a based submodule of $M$ with crystal basis $\mathcal{B}_M[\geq b]$.
  Moreover, there exists a based $\mathbf{U}$-module isomorphism
  $$
  M[\geq b]/M[> \operatorname{wt}(b)] \rightarrow V(\operatorname{wt}(b))
  $$
  which sends $G_M(b) + M[> \operatorname{wt}(b)]$ to $v_{\operatorname{wt}(b)}$.
\end{lem}

\begin{proof}
  The assertion is obvious from the above.
\end{proof}
  
\begin{lem}\label{lem: bar-invariance of v_b}
Let $(M, \mathbf{B}_M) \in \mathcal{C}^+$ and $b \in \mathcal{B}_M^{\text{hi}}$.
Then, we have $\overline{v_b} = v_b$.
\end{lem}

\begin{proof}
  Set $\lambda := \operatorname{wt}(b)$ and $u_b := G_M(b) - v_b \in M[> \lambda]$.
  Since the bar-involution on $M$ preserves $I_\mu(M)$ for all $\mu \in X^+$, we see that $\overline{v_b} \in I_\lambda(M)$ and $\overline{u_b} \in M[> \lambda]$.
  Hence, we obtain
  $$
  \overline{v_b} - v_b = u_b - \overline{u_b} \in I_\lambda(M) \cap M[> \lambda] = 0.
  $$
  This proves the assertion.
\end{proof}

\begin{lem}\label{lem: can lifting of cry morph}
Let $(M, \mathbf{B}_M), (N, \mathbf{B}_N) \in \mathcal{C}^+$, and $\phi: \mathcal{B}_M \rightarrow \mathcal{B}_N$ be a strict crystal morphism.
Then, there exists a unique $\mathbf{U}$-module homomorphism $f: M \rightarrow N$ such that $f(v_b) = v_{\phi(b)}$ for all $b \in \mathcal{B}_M^{\text{hi}}$. Here, we set $v_{\phi(b)} = 0$ if $\phi(b) = 0$.
Moreover, the following hold:
\begin{itemize}
  \item $f \circ \psi_M = \psi_N \circ f$, where $\psi_M,\psi_N$ denote the bar-involutions on $M,N$, respectively.
  \item $f(\mathcal{L}_M) \subset \mathcal{L}_N$.
  \item $\phi \circ \mathrm{ev}_\infty = \mathrm{ev}_\infty \circ f$ on $\mathcal{L}_M$.
\end{itemize}
\end{lem}

\begin{proof}
  The existence of $f$ follows from easy observation that $v_{\phi(b)}$ is either $0$ or a highest weight vector of weight $\operatorname{wt}(b)$.
  The commutativity of $f$ and the bar-involutions follows from Lemma \ref{lem: bar-invariance of v_b}.

  In order to show the remaining assertions, let $b \in \mathcal{B}_M^{\text{hi}}$ and $b' \in C(b)$.
  Then, we can write $b' = \tilde{F}_{i_1} \cdots \tilde{F}_{i_r} b$ for some $i_1,\ldots,i_r \in I$.
  Set $v := \tilde{F}_{i_1} \cdots \tilde{F}_{i_r} v_b \in \mathcal{L}_M$.
  Then, we have
  \begin{align*}
    &\mathrm{ev}_\infty(v) = \tilde{F}_{i_1} \cdots \tilde{F}_{i_r} b = b', \\
    &f(v) = \tilde{F}_{i_1} \cdots \tilde{F}_{i_r} v_{\phi(b)} \in \mathcal{L}_N, \\
    &\mathrm{ev}_\infty(f(v)) = \tilde{F}_{i_1} \cdots \tilde{F}_{i_r} \phi(b) = \phi(b').
  \end{align*}
  These imply the remaining assertions.
  Thus, the proof completes.
\end{proof}

\begin{prop}\label{prop: suff cond to be based hom}
Let $(M, \mathbf{B}_M), (N, \mathbf{B}_N) \in \mathcal{C}^+$, $\phi: \mathcal{B}_M \rightarrow \mathcal{B}_N$ be a strict crystal morphism, and $f: M \rightarrow N$ the $\mathbf{U}$-module homomorphism in {\rm Lemma \ref{lem: can lifting of cry morph}}.
Suppose that $f(E_i G_M(b)) = E_i G_N(\phi(b))$ for all $i \in I$ and $b \in \mathcal{B}_M^{\text{hi}}$.
Then, we have $f(G_M(b)) = G_N(\phi(b))$ for all $b \in \mathcal{B}_M$.
Furthermore, if $\phi(b_1) = \phi(b_2) \neq 0$ implies $b_1 = b_2$ for all $b_1,b_2 \in \mathcal{B}_M^{\text{hi}}$, then $f$ is based.
\end{prop}

\begin{proof}
  Let $b \in \mathcal{B}_M^{\text{hi}}$ and set $\lambda := \operatorname{wt}(b)$.
  Then, the number $D(b)$ of elements $b_1 \in \mathcal{B}_M^{\text{hi}}$ such that $\operatorname{wt}(b_1) > \lambda$ is finite.
  We prove that $f(G_M(b')) = G_N(\phi(b'))$ for all $b' \in C(b)$ by induction on $D(b)$.

  Assume that our claim is true for all $b_1 \in \mathcal{B}_M^{\text{hi}}$ with $D(b_1) < D(b)$; note that we assume nothing when $D(b) = 0$.
  Then, we have
  \begin{align}
    f(G_M(b'')) = G_N(\phi(b'')) \label{eq: consequence of induction hypo}
  \end{align}
  for all $b'' \in \mathcal{B}[> \lambda]$.

  Set
  $$
  u_b := G_M(b)-v_b \in M[> \lambda], \quad u_{\phi(b)} := G_N(\phi(b))-v_{\phi(b)} \in N[> \lambda].
  $$
  We shall show that $f(u_b) = u_{\phi(b)}$.
  By the assumption on $f$, we have
  $$
  E_i(f(u_b) - u_{\phi(b)}) = f(E_iG_M(b)) - E_iG_N(\phi(b)) = 0
  $$
  for all $i \in I$.
  This implies that $f(u_b)-u_{\phi(b)}$ is either $0$ or a highest weight vector of weight $\lambda$.
  However, the submodule $N[> \lambda]$, to which $f(u_b)-u_{\phi(b)}$ belongs, has no highest weight vector of weight $\lambda$.
  Hence, the equality $f(u_b) = u_{\phi(b)}$ follows.

  Now, we have
  \begin{align}
    f(G_M(b)) = f(v_b + u_b) = v_{\phi(b)} + u_{\phi(b)} = G_N(\phi(b)). \label{eq: f(G_M(b)) = G_N(phi(b))}
  \end{align}
  Let $b' \in C(b)$.
  We shall show that $f(G_M(b')) = G_N(\phi(b'))$.
  By Lemma \ref{lem: filtration on based mod} (see also Subsection \ref{subsec: V(pm lm)}), we have
  $$
  G_M(b') = G_\infty(\iota_b(b')) G_M(b) + \sum_{b'' \in \mathcal{B}_M[> \lambda]} c_{b''} G_M(b'')
  $$
  for some $b'' \in \mathcal{A}$.
  By equations \eqref{eq: consequence of induction hypo} and \eqref{eq: f(G_M(b)) = G_N(phi(b))}, we obtain
  $$
  f(G_M(b')) = G_\infty(\iota_b(b')) G_N(\phi(b)) + \sum_{b'' \in \mathcal{B}_M[> \lambda]} c_{b''} G_N(\phi(b'')).
  $$
  This implies that $f(G_M(b')) \in {}_{\mathcal{A}} N$.
  By Lemma \ref{lem: can lifting of cry morph}, we see that
  $$
  f(G_M(b')) \in \mathcal{L}_N,\ \mathrm{ev}_\infty(f(G_M(b'))) = \phi(b'),\ \overline{f(G_M(b'))} = f(G_M(b')).
  $$
  These imply that $f(G_M(b')) \in \mathcal{L}_N \cap {}_{\mathcal{A}}N \cap \overline{\mathcal{L}_N}$, and hence,
  $$
  f(G_M(b')) = (G_N \circ \mathrm{ev}_\infty)(f(G_M(b'))) = G_N(\phi(b')),
  $$
  as desired.

  The remaining assertion follows from Lemma \ref{lem: criterion of based hom}.
  Thus, the proof completes.
\end{proof}

\subsection{Based submodule $V(w\lambda, \mu)$}
For each $\lambda,\mu \in X^+$ and $w \in W$, let $V(w\lambda, \mu) \subset V(\lambda) \otimes V(\mu)$ denote the $\mathbf{U}$-submodule generated by $v_{w \lambda} \otimes v_\mu$.
By \cite[Theorem 2.2]{BW21}, the $V(w\lambda, \mu)$ is a based submodule of $V(\lambda, \mu)$.
Set
\begin{align*}
  &\mathcal{L}(w\lambda, \mu) := \mathcal{L}(\lambda, \mu) \cap V(w\lambda, \mu), \\
  &\mathcal{B}(w\lambda, \mu) := \{ v +  q^{-1} \mathcal{L}(w\lambda, \mu) \mid v \in \mathbf{B}(\lambda, \mu) \cap V(w\lambda, \mu) \}.
\end{align*}
Then, $(\mathcal{L}(w\lambda, \mu), \mathcal{B}(w\lambda, \mu))$ forms a crystal base of $V(w\lambda, \mu)$.

\begin{lem}\label{lem: V(nu) otimes V(wlm,mu) is based}
Let $\lambda,\mu,\nu \in X^+$ and $w \in W$.
Then, $V(\nu) \otimes V(w\lambda,\mu)$ is a based submodule of $V(\nu,\lambda,\mu)$ $($see {\rm Example \ref{ex: based mods}} \eqref{ex: based mod V(lm1,...,lmr)} for the based module structure of $V(\nu,\lambda,\mu)$$)$.
\end{lem}

\begin{proof}
  It suffices to show that $G(b_1 \otimes b_2 \otimes b_3) \in V(\nu) \otimes V(w\lambda,\mu)$ for all $b_1 \otimes b_2 \otimes b_3 \in \mathcal{B}(\nu) \otimes \mathcal{B}(w\lambda, \mu) \subset \mathcal{B}(\nu,\lambda,\mu)$.
  By the construction of $G(b_1 \otimes b_2 \otimes b_3)$ in \cite[Theorem 2.9]{BW16}, we have
  $$
  G(b_1 \otimes b_2 \otimes b_3) = \sum_{b'_1 \otimes b'_2 \otimes b'_3 \in \mathcal{B}(\nu, \lambda, \mu)} c_{b'_1,b'_2,b'_3} G(b'_1) \otimes G(b'_2 \otimes b'_3)
  $$
  for some $c_{b'_1,b'_2,b'_3} \in \mathcal{A}$ such that $c_{b'_1,b'_2,b'_3} = 0$ unless $G(b'_2 \otimes b'_3)$ belongs to the smallest based $\mathbf{U}$-submodule of $V(\lambda,\mu)$ containing $G(b_2 \otimes b_3)$.
  Since $V(w\lambda, \mu)$ is based submodule of $V(\lambda,\mu)$ containing $G(b_2 \otimes b_3)$, we see that $c_{b'_1,b'_2,b'_3} = 0$ unless $G(b'_2 \otimes b'_3) \in V(w\lambda, \mu)$.
  This proves the assertion.
\end{proof}

\section{$\imath$Quantum group}\label{sec: iqua grp}
In this section, after recalling the notion of $\imath$quantum groups and based $\mathbf{U}^\imath$-modules, we prove the existence of certain based $\mathbf{U}$-module homomorphisms in Proposition \ref{prop: existence of based U-hom}.
This reduces the problem of stability of $\imath$canonical bases to that of existence of based $\mathbf{U}^\imath$-module homomorphism $\delta^\imath_\nu: V(\nu+w_\bullet\tau\nu) \rightarrow \mathbb{Q}(q)$ sending the highest weight vector to $1$ for each $\nu \in X^+$.

\subsection{Admissible pair}\label{subsec: adm pair}
An admissible pair \cite[Definition 2.3]{Kol14} is a pair $(I_\bullet, \tau)$ consisting of a subset $I_\bullet \subset I$ of finite type and a Dynkin diagram automorphism $\tau$ on $I$ satisfying the following:
\begin{itemize}
  \item $\tau^2 = \mathrm{id}$.
  \item $w_\bullet(\alpha_j) = -\alpha_{\tau(j)}$ for all $j \in I_\bullet$, where $w_\bullet$ denotes the longest element of the Weyl group $W_{I_\bullet}$ for $I_\bullet$.
  \item $\langle \rho^\vee_\bullet, \alpha_i \rangle \in \mathbb{Z}$ for all $i \in I_\circ := I \setminus I_\bullet$ with $\tau(i) = i$, where $\rho^\vee_\bullet$ denotes half the sum of positive coroots for $I_\bullet$.
\end{itemize}
In order to clarify which Dynkin diagram $I$ is considered, we sometimes denote an admissible pair by the triple $(I, I_\bullet, \tau)$.

An admissible pair $(I_\bullet, \tau)$ is said to be \emph{irreducible} if for each $i,j \in I$, there exists a sequence $i = i_1,\ldots,i_r = j \in I$ such that for each $k = 1,\ldots,r-1$, we have either $a_{i_k,i_{k+1}} \neq 0$ or $i_{k+1} = \tau(i_k)$.

An admissible pair is said to be of \emph{finite type} if the Dynkin diagram $I$ is of finite type.
The admissible pairs of finite type with $I_\bullet \neq I$ are identical to the Satake diagrams in \cite{Ara62} (cf. \cite[after Definition 2.3]{Kol14}).

The number of $\tau$-orbits in $I_\circ$ is called the \emph{real rank} of $(I_\bullet, \tau)$.

Later, we will restrict our attention to the irreducible admissible pairs of finite type of real rank $1$.
Here is the list of such admissible pairs; for the labels of Dynkin diagrams (except type $A_1 \times A_1$), we follow \cite[Figs. 2.5 and 2.6]{BS17}:
\begin{itemize}
  \item Type AI.
  \begin{itemize}
    \item $I = \{1\}$: type $A_1$.
    \item $I_\bullet = \emptyset$.
    \item $\tau = \mathrm{id}$.
  \end{itemize}
  \item Type AII.
  \begin{itemize}
    \item $I = \{1,2,3\}$: type $A_3$.
    \item $I_\bullet = \{1,3\}$.
    \item $\tau = \mathrm{id}$.
  \end{itemize}
  \item Type AIII.
  \begin{itemize}
    \item $I = \{1,2\}$: type $A_1 \times A_1$.
    \item $I_\bullet = \emptyset$.
    \item $\tau(1) = 2,\ \tau(2) = 1$.
  \end{itemize}
  \item Type AIV.
  \begin{itemize}
    \item $I = \{1,\ldots,n\}$: type $A_n$ with $n \geq 2$.
    \item $I_\bullet = \{2,\ldots,n-1\}$.
    \item $\tau(i) = n-i+1$.
  \end{itemize}
  \item Type BII.
  \begin{itemize}
    \item $I = \{1,\ldots,n\}$: type $B_n$ with $n \geq 2$.
    \item $I_\bullet = \{2,\ldots,n\}$.
    \item $\tau = \mathrm{id}$.
  \end{itemize}
  \item Type CII.
  \begin{itemize}
    \item $I = \{1,\ldots,n\}$: type $C_n$ with $n \geq 3$.
    \item $I_\bullet = \{1,3,\ldots,n\}$.
    \item $\tau = \mathrm{id}$.
  \end{itemize}
  \item Type DII.
  \begin{itemize}
    \item $I = \{1,\ldots,n\}$: type $D_n$ with $n \geq 4$.
    \item $I_\bullet = \{2,\ldots,n\}$.
    \item $\tau(i) = \begin{cases}
      n & \text{ if } i = n-1 \text{ and } n \in 2\mathbb{Z}, \\
      n-1 & \text{ if } i = n \text{ and } n \in 2\mathbb{Z}, \\
      i & \text{ otherwise}.
    \end{cases}$
  \end{itemize}
  \item Type FII.
  \begin{itemize}
    \item $I = \{1,2,3,4\}$: type $F_4$.
    \item $I_\bullet = \{1,2,3\}$.
    \item $\tau = \mathrm{id}$.
  \end{itemize}
\end{itemize}

\subsection{$\imath$Quantum group}
Assume that $\tau$ induces involutive automorphisms on $Y$ and $X$ such that
$$
\tau h_i = h_{\tau(i)},\ \tau \alpha_i = \alpha_{\tau(i)},\ \langle \tau h, \tau \lambda \rangle = \langle h,\lambda \rangle
$$
for all $i \in I$, $h \in Y$, $\lambda \in X$.
Note that this assumption is always satisfied when $I$ is of finite type and $Y = \sum_{i \in I} \mathbb{Z} h_i$.

Set
$$
Y^\imath := \{ h \in Y \mid h + w_\bullet \tau h = 0 \},\quad X^\imath := X/\{ \lambda + w_\bullet \tau \lambda \mid \lambda \in X \}.
$$
For each $\lambda \in X$, let $\bar{\lambda} \in X^\imath$ denote the image of $\lambda$.
The perfect pairing $\langle,\rangle: Y \times X \rightarrow \mathbb{Z}$ induces a bilinear map $\langle,\rangle: Y^\imath \times X^\imath \rightarrow \mathbb{Z}$ such that
$$
\langle h, \bar{\lambda} \rangle = \langle h, \lambda \rangle
$$
for all $h \in Y^\imath$, $\lambda \in X$.

For each $i \in I_\circ$, chose $\varsigma_i \in \mathbb{Z}[q,q^{-1}]^\times$ and $\kappa_i \in \mathbb{Z}[q,q^{-1}]$ in a way such that
\begin{itemize}
  \item $\varsigma_i = \varsigma_{\tau(i)}$ if $\langle h_i, w_\bullet\alpha_{\tau(i)} \rangle = 0$.
  \item $\kappa_i = 0$ unless $\tau(i) = i$, $\langle h_i,\alpha_j \rangle = 0$ for all $j \in I_\bullet$, and $\langle h_k,\alpha_i \rangle \in 2\mathbb{Z}$ for all $k \in I_\circ$ with $\tau(k) = k$ and $\langle h_k,\alpha_j \rangle = 0$ for all $j \in I_\bullet$.
  \item $\varsigma_{\tau(i)} = (-1)^{\langle 2\rho_\bullet^\vee, \alpha_i \rangle} q_i^{-\langle h_i, 2\rho_\bullet + w_\bullet \alpha_{\tau(i)} \rangle} \overline{\varsigma_i}$, where $\rho_\bullet$ denotes half the sum of positive roots for $I_\bullet$.
  \item $\overline{\kappa_i} = \kappa_i$.
\end{itemize}
The first two conditions are needed to make the $\imath$quantum group below to be a reasonable size (cf. \cite[Theorem 10.8]{Kol14}, see also \cite[Remark 3.3]{BK15}).
The third one guarantees the existence of the bar-involution on $\mathbf{U}^\imath$ (cf. \cite[Corollary 4.2]{Kol21}).
The fourth one was used to construct the quasi-$K$-matrix in \cite[Theorem 6.10]{BK19}.
The constraint $\varsigma_i,\kappa_i \in \mathbb{Z}[q,q^{-1}]$ is necessary for the theory of $\imath$canonical bases in \cite{BW21}.

The $\imath$quantum group is the subalgebra $\mathbf{U}^\imath \subset \mathbf{U}$ generated by
$$
\{ E_j, B_i, K_h \mid j \in I_\bullet,\ i \in I,\ h \in Y^\imath \},
$$
where
$$
B_i := \begin{cases}
  F_i & \text{ if } i \in I_\bullet, \\
  F_i + \varsigma_i T_{w_\bullet}(E_{\tau(i)}) K_i^{-1} + \kappa_i K_i^{-1} & \text{ if } i \in I_\circ.
\end{cases}
$$

By \cite[Corollary 4.2]{Kol21}, there exists a unique $\mathbb{Q}$-algebra automorphism $\overline{\cdot}$ on $\mathbf{U}^\imath$ such that
$$
\overline{E_j} = E_j,\ \overline{B_i} = B_i,\ \overline{K_h} = K_{-h},\ \bar{q} = q^{-1}
$$
for all $j \in I_\bullet$, $i \in I$, $h \in Y^\imath$.
We call it the $\imath$bar-involution on $\mathbf{U}^\imath$.

Let $\mathbf{U}_{I_\bullet} \subset \mathbf{U}^\imath$ denote the subalgebra generated by $\{ E_j,F_j,K_j^{\pm 1} \mid j \in I_\bullet \}$.
It is the quantum group associated with $I_\bullet$.
Similarly, let $\mathbf{U}_{I_\bullet}^+$, $\mathcal{B}_{I_\bullet}(-\infty)$, $\mathcal{B}_{I_\bullet}(-\infty;-\lambda)$, $\mathcal{B}_{I_\bullet}(-\lambda)$, and so on denote the same things without subscripts, but associated with $I_\bullet$.

\subsection{Based module}
Let $\dot{\mathbf{U}}^\imath = \bigoplus_{\zeta,\eta \in X^\imath} 1_\zeta \dot{\mathbf{U}}^\imath 1_\eta$ denote the modified $\imath$quantum group (cf. \cite[Section 3.5]{BW21}, \cite[Section 3.3]{W21a}), and ${}_{\mathcal{A}}\dot{\mathbf{U}}^\imath$ its $\mathcal{A}$-form.

Following \cite[Definition 6.11]{BW21}, we define a based $\mathbf{U}^\imath$-module to be a pair $(M, \mathbf{B}^\imath_M)$ consisting of a weight $\mathbf{U}^\imath$-module $M$ and its linear basis $\mathbf{B}^\imath_M$ satisfying the following:
\begin{itemize}
  \item $\mathbf{B}^\imath_M \cap M_\zeta$ forms a basis of $M_\zeta$, where
  $$
  M_\zeta := \{ m \in M \mid K_h m = q^{\langle h,\zeta \rangle} \text{ for all } h \in Y^\imath \}.
  $$
  \item ${}_{\mathcal{A}} M := \mathcal{A} \mathbf{B}^\imath_M$ is a ${}_{\mathcal{A}} \dot{\mathbf{U}}^\imath$-submodule. We call it the $\mathcal{A}$-form of $M$.
  \item The $\mathbb{Q}$-linear map $\overline{\cdot}: M \rightarrow M$ sending $q^n v$ to $q^{-n} v$ for all $n \in \mathbb{Z}$ and $v \in \mathbf{B}^\imath_M$ satisfies $\overline{xm} = \bar{x}\bar{m}$ for all $x \in \mathbf{U}^\imath$ and $m \in M$.
  We call it the $\imath$bar-involution on $M$.
  \item Setting $\mathcal{L}_M := \mathbf{A}_\infty \mathbf{B}^\imath_M$, the quotient map $\mathrm{ev}_\infty: \mathcal{L}_M \rightarrow \mathcal{L}_M/q^{-1}\mathcal{L}_M$ restricts to a $\mathbb{Q}$-linear isomorphism $\mathcal{L}_M \cap {}_{\mathcal{A}}M \cap \overline{\mathcal{L}_M} \rightarrow \mathcal{L}_M/q^{-1}\mathcal{L}_M$; let $G^\imath_M$ denote its inverse.
\end{itemize}
The notions of based $\mathbf{U}^\imath$-submodules and based $\mathbf{U}^\imath$-module homomorphisms are defined in the same way as based $\mathbf{U}$-modules.
The following can be proved by the same way as Lemma \ref{lem: criterion of based hom}.

\begin{lem}\label{lem: criterion of based ihom}
  Let $(M, \mathbf{B}_M), (N, \mathbf{B}_N)$ be based $\mathbf{U}^\imath$-modules, and $f: M \rightarrow N$ a $\mathbf{U}^\imath$-module homomorphism.
  Then, $f$ is a based $\mathbf{U}^\imath$-module homomorphism if and only if it satisfies the following:
  \begin{itemize}
    \item $f(\mathcal{L}_M) \subset \mathcal{L}_N$; it induces a map
    $$
    \phi: \mathcal{B}_M \rightarrow \mathcal{B}_N;\ b \mapsto \mathrm{ev}_\infty(f(G_M(b))).
    $$
    \item $f({}_{\mathcal{A}}M) \subset {}_{\mathcal{A}}N$.
    \item $f \circ \psi^\imath_M = \psi^\imath_N \circ f$, where $\psi^\imath_M,\psi^\imath_N$ denote the $\imath$bar-involution on $M,N$, respectively.
    \item $\phi$ is injective on $\{ b \in \mathcal{B}_M \mid \phi(b) \neq 0 \}$.
  \end{itemize}
\end{lem}

\begin{ex}\normalfont\label{ex: based Ui-mod and hom}
  Let $\lambda,\mu,\lambda_1,\ldots,\lambda_r \in X^+$ and $w \in W$.
  \begin{enumerate}
    \item Let $\mathbf{B}^\imath(\lambda_1,\ldots,\lambda_r)$ denote the $\imath$canonical basis of the based $\mathbf{U}$-module $V(\lambda_1,\ldots,\lambda_r)$ in the sense of \cite[Theorem 6.12]{BW21}.
    Then, $(V(\lambda_1,\ldots,\lambda_r), \mathbf{B}^\imath(\lambda_1,\ldots,\lambda_r))$ forms a based $\mathbf{U}^\imath$-module.
    \item By \cite[Theorem 6.13 (2)]{BW21}, $V(w\lambda,\mu)$ is a based $\mathbf{U}^\imath$-submodule of $V(\lambda,\mu)$; set $\mathbf{B}^\imath(w\lambda, \mu) := \mathbf{B}^\imath(\lambda,\mu) \cap V(w\lambda, \mu)$.
    \item\label{ex: based Ui-mod and hom, Uidot} Let $\dot{\mathbf{B}}^\imath$ denote the $\imath$canonical basis of $\dot{\mathbf{U}}^\imath$ in the sense of \cite[Theorem 7.2]{BW21}.
    Then, $(\dot{\mathbf{U}}^\imath, \dot{\mathbf{B}}^\imath)$ is a based $\mathbf{U}^\imath$-module.
  \end{enumerate}
\end{ex}

\begin{lem}
Let $\lambda,\mu,\nu \in X^+$ and $w \in W$.
Then, $V(\nu) \otimes V(w\lambda,\mu)$ is a based $\mathbf{U}^\imath$-submodule of $V(\nu,\lambda,\mu)$.
\end{lem}

\begin{proof}
  It suffices to show that $G^\imath(b_1 \otimes b_2 \otimes b_3) \in V(\nu) \otimes V(w\lambda,\mu)$ for all $b_1 \otimes b_2 \otimes b_3 \in \mathcal{B}(\nu) \otimes \mathcal{B}(w\lambda, \mu) \subset \mathcal{B}(\nu,\lambda,\mu)$.
  By the construction of $G^\imath(b_1 \otimes b_2 \otimes b_3)$ in \cite[Theorem 6.12]{BW21}, we have
  $$
  G^\imath(b_1 \otimes b_2 \otimes b_3) = \sum_{b'_1 \otimes b'_2 \otimes b'_3 \in \mathcal{B}(\nu,\lambda,\mu)} c_{b'_1,b'_2,b'_3} G(b'_1 \otimes b'_2 \otimes b'_3)
  $$
  for some $c_{b'_1,b'_2,b'_3} \in \mathcal{A}$ such that $c_{b'_1,b'_2,b'_3} = 0$ unless $G(b'_1 \otimes b'_2 \otimes b'_3)$ belongs to the smallest based $\mathbf{U}$-submodule of $V(\nu,\lambda,\mu)$ containing $G(b_1 \otimes b_2 \otimes b_3)$.
  Since $V(\nu) \otimes V(w\lambda, \mu)$ is a based $\mathbf{U}$-submodule containing $G(b_1 \otimes b_2 \otimes b_3)$, we see that $c_{b'_1,b'_2,b'_3} = 0$ unless $G(b'_1 \otimes b'_2 \otimes b'_3) \in V(\nu) \otimes V(w\lambda, \mu)$.
  This proves the assertion.
\end{proof}

\subsection{Based submodule $V(w_\bullet \lambda, \mu)$}
Let $\lambda,\mu \in X^+$.
Let $C_{I_\bullet}(b_\lambda) \subset \mathcal{B}(\lambda)$ denote the connected component of $\mathcal{B}(\lambda)$ containing $b_\lambda$ as the crystal of type $I_\bullet$.
Namely,
$$
C_{I_\bullet}(b_\lambda) = \{ \tilde{F}_{j_1} \cdots \tilde{F}_{j_r} b_\lambda \mid j_1,\ldots,j_r \in I_\bullet \} \setminus \{0\}.
$$
Note that it is isomorphic to $\mathcal{B}_{I_\bullet}(w_\bullet \lambda)$ as a crystal of type $I_\bullet$; we regard $-w_\bullet \lambda$ as a dominant weight for $I_\bullet$.

\begin{lem}\label{lem: characterization of C_dot(b_lm)}
We have
$$
C_{I_\bullet}(b_\lambda) = \{ b \in \mathcal{B}(\lambda) \mid \operatorname{wt}(b) \geq w_\bullet \lambda \}.
$$
\end{lem}

\begin{proof}
  Since $C_{I_\bullet}(b_\lambda) \simeq \mathcal{B}_{I_\bullet}(w_\bullet\lambda)$, we see that $\operatorname{wt}(b) \geq w_\bullet \lambda$ for all $b \in C_{I_\bullet}(b_\lambda)$.
  This proves the containment ``$\subset$''.

  We shall prove the opposite direction.
  Let $b \in \mathcal{B}(\lambda)$ be such that $\operatorname{wt}(b) \geq w_\bullet \lambda$.
  Since $\operatorname{wt}(b) \leq \lambda$, we can write
  \begin{align}
    \lambda - \operatorname{wt}(b) = \sum_{j \in I_\bullet} m_j \alpha_j \label{eq: lm - wt(b)}
  \end{align}
  for some $m_j \geq 0$.
  On the other hand, we have
  $$
  b = \tilde{F}_{i_1} \cdots \tilde{F}_{i_r} b_\lambda
  $$
  for some $i_1,\ldots,i_r \in I$.
  Taking equation \eqref{eq: lm - wt(b)} into account, we see that $i_1,\ldots,i_r \in I_\bullet$.
  This implies that $b \in C_{I_\bullet}(b_\lambda)$.
  Thus, the proof completes.
\end{proof}

By Subsection \ref{subsec: V(pm lm)}, there exist maps
$$
\pi = \pi_{I_\bullet;w_\bullet\lambda}: \mathcal{B}_{I_\bullet}(-\infty) \rightarrow C_{I_\bullet}(b_\lambda) \sqcup \{0\},\ \iota = \iota_{I_\bullet; w_\bullet\lambda}: C_{I_\bullet}(b_\lambda) \hookrightarrow \mathcal{B}_{I_\bullet}(-\infty; w_\bullet\lambda)
$$
such that
\begin{align}
\begin{split}
  &\pi \circ \iota = \mathrm{id}, \\
  &\iota(b_{w_\bullet\lambda}) = b_{-\infty}, \\
  &\operatorname{wt}(\iota(b)) = \operatorname{wt}(b) - w_\bullet\lambda, \\
  &\varphi_j(\iota(b)) = \varphi_j(b), \\
  &\tilde{E}_j \iota(b) = \iota(\tilde{E}_j b) \text{ if } \tilde{E}_j b \neq 0
\end{split} \label{eq: iota_dot lm}
\end{align}
for all $j \in I_\bullet$ and $b \in C_{I_\bullet}(b_\lambda)$.

\begin{lem}\label{lem: G(b otimes b_mu)}
  Let $b \in \mathcal{B}_{I_\bullet}(-\infty)$ and set $\pi := \pi_{I_\bullet; w_\bullet\lambda}$.
  Then, we have
  $$
  G(b)(v_{w_\bullet\lambda} \otimes v_\mu) = G(\pi(b)) \otimes v_\mu = G(\pi(b) \otimes b_\mu).
  $$
  In particular, $G(b' \otimes b_\mu) = G(b') \otimes v_\mu \in V(w_\bullet\lambda, \mu)$ for all $b' \in C_{I_\bullet}(b_\lambda)$.
\end{lem}

\begin{proof}
  We have
  $$
  G(b)(v_{w_\bullet\lambda} \otimes v_\mu) = G(b)v_{w_\bullet\lambda} \otimes v_\mu = G(\pi(b)) \otimes v_\mu.
  $$
  The left-hand side belongs to ${}_{\mathcal{A}} V(w_\bullet\lambda,\mu)$, and is bar-invariant.
  On the other hand, the right-hand side belongs to $\mathcal{L}(\lambda,\mu)$, and its image under $\mathrm{ev}_\infty$ is $\pi(b) \otimes b_\mu$.
  Therefore, we have $G(\pi(b)) \otimes v_\mu \in \mathcal{L}(w_\bullet\lambda,\mu) \cap {}_{\mathcal{A}} V(w_\bullet\lambda,\mu) \cap \overline{\mathcal{L}(w_\bullet\lambda,\mu)}$, and consequently,
  $$
  G(\pi(b)) \otimes v_\mu = (G \circ \mathrm{ev}_\infty)(G(\pi(b)) \otimes v_\mu) = G(\pi(b) \otimes b_\mu),
  $$
  as desired.
\end{proof}

\begin{prop}\label{prop: hwe in B(w_dot lm, mu)}
Let $\lambda, \mu \in X^+$ and $b \in \mathcal{B}(w_\bullet \lambda, \mu)$.
Then, $b$ is a highest weight element if and only if $b = b_1 \otimes b_\mu$ for some $b_1 \in C_{I_\bullet}(b_\lambda)$ such that $\varepsilon_i(b_1) \leq \langle h_i,\mu \rangle$ for all $i \in I$.
\end{prop}

\begin{proof}
  Let $b \in \mathcal{B}(w_\bullet \lambda, \mu)^{\text{hi}}$.
  Since the $\mathbf{U}$-module $V(w_\bullet \lambda,\mu)$ is generated by a global crystal basis element $v_{w_\bullet \lambda} \otimes v_\mu$ of weight $w_\bullet \lambda + \mu$, we have $I_{\lambda'}(V(w_\bullet \lambda, \mu)) = 0$ unless $\lambda' \geq w_\bullet \lambda + \mu$.
  Hence, we obtain $\operatorname{wt}(b) \geq w_\bullet \lambda+\mu$.
  Now, the ``only if'' part follows from Lemmas \ref{lem: hw element in tensor crystal} and \ref{lem: characterization of C_dot(b_lm)}.

  Let us prove the opposite direction.
  Let $b_1 \in C_{I_\bullet}(b_\lambda)$ be such that $\varepsilon_i(b) \leq \langle h_i,\mu \rangle$ for all $i \in I$.
  By Lemma \ref{lem: hw element in tensor crystal}, we have $b := b_1 \otimes b_\mu \in \mathcal{B}(\lambda,\mu)^{\text{hi}}$.
  By Lemma \ref{lem: G(b otimes b_mu)}, we see that $G(b) \in V(w_\bullet \lambda, \mu)$.
  Hence, we obtain
  $$
  b = \mathrm{ev}_\infty(G(b)) \in \mathcal{B}(w_\bullet\lambda,\mu).
  $$
  This completes the proof.
\end{proof}

For each $\lambda, \mu \in X^+$, set
\begin{align*}
  &C_{I_\bullet; \mu}(b_\lambda) := \{ b \in C_{I_\bullet}(b_\lambda) \mid \varepsilon_i(b) \leq \langle h_i,\mu \rangle \ \text{ for all } i \in I \}.
\end{align*}
Then, by Proposition \ref{prop: hwe in B(w_dot lm, mu)}, we obtain
\begin{align}
  \mathcal{B}(w_\bullet\lambda, \mu)^{\text{hi}} = \{ b \otimes b_\mu \mid b \in C_{I_\bullet;\mu}(b_\lambda) \}. \label{eq: hwe in B(w_dot lm, mu)}
\end{align}

For each $\nu \in X^+$, consider the compositions
\begin{align*}
  &\pi = \pi_{I_\bullet;w_\bullet\lambda,\nu}: C_{I_\bullet}(b_{\lambda+\tau\nu}) \xrightarrow{\iota_{I_\bullet;w_\bullet(\lambda+\tau\nu)}} \mathcal{B}_{I_\bullet}(-\infty; w_\bullet(\lambda+\tau\nu)) \xrightarrow{\pi_{I_\bullet;w_\bullet\lambda}} C_{I_\bullet}(b_\lambda) \sqcup \{0\}, \\
  &\iota = \iota_{I_\bullet;w_\bullet\lambda,\nu}: C_{I_\bullet}(b_\lambda) \xrightarrow{\iota_{I_\bullet; w_\bullet\lambda}} \mathcal{B}_{I_\bullet}(-\infty;w_\bullet\lambda) \xrightarrow{\pi_{I_\bullet; w_\bullet(\lambda+\tau\nu)}} C_{I_\bullet}(b_{\lambda+\tau\nu}).
\end{align*}
They satisfy the following (cf. equation \eqref{eq: iota_dot lm}):
\begin{align}
\begin{split}
  &\pi \circ \iota = \mathrm{id}, \\
  &\iota(b_{w_\bullet \lambda}) = b_{w_\bullet(\lambda + \tau\nu)}, \\
  &\varphi_j(\iota(b)) = \varphi_j(b), \\
  &\operatorname{wt}(\iota(b)) = \operatorname{wt}(b) + w_\bullet \tau\nu, \\
  &\tilde{E}_j \iota(b) = \iota(\tilde{E}_j b) \text{ if } \tilde{E}_j b \neq 0
\end{split} \label{eq: iota_dot lm nu}
\end{align}
for all $j \in I_\bullet$ and $b \in C_{I_\bullet}(b_\lambda)$.

\begin{lem}\label{lem: C_Idot(b_lm) and iota}
Let $\lambda,\mu,\nu \in X^+$ and $b \in C_{I_\bullet}(b_\lambda)$.
Set $\iota := \iota_{I_\bullet; w_\bullet\lambda,\nu}$.
Then, we have $\iota(b) \in C_{I_\bullet; \mu+\nu}(b_{\lambda+\tau\nu})$ if and only if $b \in C_{I_\bullet; \mu}(b_\lambda)$.
\end{lem}

\begin{proof}
  By equation \eqref{eq: iota_dot lm nu}, we have
  $$
  \varepsilon_j(\iota(b)) = \varphi_j(\iota(b))-\langle h_j,\operatorname{wt}(\iota(b)) \rangle = \varphi_j(b)-\langle h_j,\operatorname{wt}(b)+w_\bullet\tau\nu \rangle = \varepsilon_j(b) + \langle h_j,\nu \rangle
  $$
  for all $j \in I_\bullet$, and
  $$
  \varepsilon_i(\iota(b)) = 0 = \varepsilon_i(b)
  $$
  for all $i \in I_\circ$.
  Now, the assertion follows from the definition of $C_{I_\bullet;\mu}(b_\lambda)$.
\end{proof}

\begin{lem}\label{lem: nu+wdottaunu is dom}
Let $\nu \in X^+$.
Then, we have $\nu+w_\bullet\tau\nu \in X^+$.
\end{lem}

\begin{proof}
  For each $i \in I$, we have
  $$
  \langle h_i, \nu+w_\bullet\tau\nu \rangle = \langle h_i,\nu \rangle + \langle w_\bullet h_{\tau(i)}, \nu \rangle.
  $$
  Now, the assertion follows from easy observation that $w_\bullet h_{\tau(i)} = -h_i$ if $i \in I_\bullet$ and that $w_\bullet h_{\tau(i)}$ is a positive coroot if $i \in I_\circ$.
\end{proof}

By Lemma \ref{lem: nu+wdottaunu is dom}, we can apply Lemma \ref{lem: V(nu) otimes V(wlm,mu) is based} to see that $V(\nu+w_\bullet\tau\nu) \otimes V(w_\bullet\lambda,\mu)$ is a based $\mathbf{U}$-submodule of $V(\nu+w_\bullet\tau\nu, \lambda,\mu)$.

\begin{lem}\label{lem: G(b_nu otimes b otimes b_mu)}
  Let $b \in \mathcal{B}_{I_\bullet}(-\infty)$, and set $\pi := \pi_{I_\bullet; w_\bullet\lambda}$.
  Then, we have
  $$
  G(b)(v_{\nu+w_\bullet\tau\nu} \otimes v_{w_\bullet\lambda} \otimes v_\mu) = v_{\nu+w_\bullet\tau\nu} \otimes G(\pi(b)) \otimes v_\mu = G(b_{\nu+w_\bullet\tau\nu} \otimes \pi(b) \otimes b_\mu).
  $$
  In particular, $G(b_{\nu+w_\bullet\tau\nu} \otimes b' \otimes b_\mu) = v_{\nu+w_\bullet\tau\nu} \otimes G(b') \otimes v_\nu$ for all $b' \in C_{I_\bullet}(b_\lambda)$.
\end{lem}

\begin{proof}
  Noting that
  $$
  E_j v_{\nu+w_\bullet\tau\nu} = 0, \quad \langle h_j, \nu+w_\bullet\tau\nu \rangle = 0
  $$
  for all $j \in I_\bullet$, we see that the same argument as the proof of Lemma \ref{lem: G(b otimes b_mu)} proves the assertion.
\end{proof}

\begin{prop}\label{prop: existence of based U-hom}
Let $\lambda,\mu,\nu \in X^+$ and set $\iota := \iota_{I_\bullet; w_\bullet\lambda, \nu}$ and $\pi := \pi_{I_\bullet; w_\bullet\lambda, \nu}$.
Then, there exists a unique based $\mathbf{U}$-module homomorphism
$$
f: V(w_\bullet(\lambda+\tau\nu), \mu+\nu) \rightarrow V(\nu + w_\bullet\tau\nu) \otimes V(w_\bullet \lambda, \mu)
$$
such that
$$
f(G(b \otimes b_{\mu+\nu})) = G(b_{\nu+w_\bullet\tau\nu} \otimes \pi(b) \otimes b_\mu)
$$
for all $b \in C_{I_\bullet;\mu+\nu}(b_{\lambda+\tau\nu})$.
\end{prop}

\begin{proof}
  First, we observe that there exists a strict crystal morphism
  $$
  \phi: \mathcal{B}(w_\bullet(\lambda+\tau\nu), \mu+\nu) \rightarrow \mathcal{B}(\nu+w_\bullet\tau\nu) \otimes \mathcal{B}(w_\bullet\lambda,\mu)
  $$
  such that
  $$
  \phi(b \otimes b_{\mu+\nu}) = b_{\nu+w_\bullet\tau\nu} \otimes \pi(b) \otimes b_\mu
  $$
  for all $b \in C_{I_\bullet;\mu+\nu}(b_{\lambda+\tau\nu})$.
  This follows from equation \eqref{eq: hwe in B(w_dot lm, mu)}, and that $b \otimes b_{\mu+\nu}$ is a highest weight element of weight $\operatorname{wt}(b) + \mu+\nu$ and $b_{\nu+w_\bullet\tau\nu} \otimes \pi(b) \otimes b_\mu$ is either $0$ or a highest weight element of weight $\operatorname{wt}(b)+\mu+\nu$ (see Lemma \ref{lem: C_Idot(b_lm) and iota}) for all $b \in C_{I_\bullet; \mu+\nu}(b_\lambda)$.

  Let $f: V(w_\bullet(\lambda+\tau\nu), \mu+\nu) \rightarrow V(\nu+w_\bullet\tau\nu) \otimes V(w_\bullet\lambda,\mu)$ denote the $\mathbf{U}$-module homomorphism in Lemma \ref{lem: can lifting of cry morph}.
  By Proposition \ref{prop: suff cond to be based hom}, in order to prove the assertion, it suffices to show that
  $$
  f(E_i G(b)) = E_i G(\phi(b))
  $$
  for all $i \in I$ and $b \in \mathcal{B}(w_\bullet(\lambda+\tau\nu), \mu+\nu)^{\text{hi}}$ by induction on $D(b)$ (see the proof of Proposition \ref{prop: suff cond to be based hom} for the definition of $D(b)$).
  By weight consideration, we see that $E_iG(b) = 0 = E_iG(\phi(b))$ if $i \in I_\circ$.
  Hence, we only need to consider the case when $i \in I_\bullet$.

  Assume that our claim is true for all $b'$ with $D(b') < D(b)$; note that we assume nothing when $D(b) = 0$.
  Set $\zeta := \operatorname{wt}(b)$.
  Then, by Proposition \ref{prop: suff cond to be based hom}, $f$ restricts to a based $\mathbf{U}$-module homomorphism
  $$
  V(w_\bullet(\lambda+\tau\nu), \mu+\nu)[> \zeta] \rightarrow V(\nu+w_\bullet\tau\nu) \otimes V(w_\bullet\lambda,\mu).
  $$
  Let $b_1 \in C_{I_\bullet;\mu+\nu}(b_{\lambda+\tau\nu})$ be such that $b = b_1 \otimes b_{\mu+\nu}$.
  Set $\tilde{b}_1 := \iota_{I_\bullet;w_\bullet(\lambda+\tau\nu)}(b_1)$.
  Let us write
  $$
  E_i G(\tilde{b}_1) = \sum_{b' \in \mathcal{B}_{I_\bullet}(-\infty)} c_{b'} G(b')
  $$
  for some $c_{b'} \in \mathcal{A}$.  
  Then, using Lemma \ref{lem: G(b otimes b_mu)}, we compute as
  \begin{align*}
    E_i G(b) &= E_iG(b_1 \otimes b_{\mu+\nu}) \\
    &= E_i G(\tilde{b}_1) (v_{w_\bullet(\lambda+\tau\nu)} \otimes v_{\mu+\nu}) \\
    &= \sum_{b'} c_{b'}G(b') (v_{w_\bullet(\lambda+\tau\nu)} \otimes v_{\mu+\nu}) \\
    &= \sum_{b'} c_{b'} G(\pi'(b') \otimes b_{\mu+\nu}),
  \end{align*}
  where $\pi' := \pi_{I_\bullet; w_\bullet(\lambda+\tau\nu)}$.
  Since $E_i G(b) \in V(w_\bullet(\lambda+\tau\nu), \mu+\nu)[> \zeta]$, we have $c_{b'} = 0$ unless $\pi'(b') \otimes b_{\mu+\nu} \in \mathcal{B}(w_\bullet(\lambda+\tau\nu), \mu+\nu)[> \zeta]$.
  By our induction hypothesis, this implies that
  $$
  f(E_i G(b)) = \sum_{b'} c_{b'} G(\phi(\pi'(b') \otimes b_{\mu+\nu})).
  $$

  On the other hand, using Lemma \ref{lem: G(b_nu otimes b otimes b_mu)}, we compute as
  \begin{align*}
    E_i G(\phi(b)) &= E_iG(b_{\nu+w_\bullet\tau\nu} \otimes \pi(b_1) \otimes b_\mu) \\
    &= E_i G(\tilde{b}_1) (v_{\nu+w_\bullet\tau\nu} \otimes v_{w_\bullet\lambda} \otimes v_\mu) \\
    &= \sum_{b'} c_{b'} G(b') (v_{\nu+w_\bullet\tau\nu} \otimes v_{w_\bullet\lambda} \otimes v_\mu) \\
    &= \sum_{b'} c_{b'} G( b_{\nu+w_\bullet\tau\nu} \otimes \pi''(b') \otimes b_\mu),
  \end{align*}
  where $\pi'' := \pi_{I_\bullet; w_\bullet\lambda}$.
  Hence, in order to complete the proof, it suffices to show that $\phi(\pi'(b') \otimes b_{\mu+\nu}) = b_{\nu+w_\bullet\tau\nu} \otimes \pi''(b') \otimes b_\mu$ for all $b'$ with $c_{b'} \neq 0$.

  Let $b' \in \mathcal{B}_{I_\bullet}(-\infty)$ be such that $c_{b'} \neq 0$.
  Suppose first that $\pi'(b') = 0$.
  Then, we have $\pi''(b') = \pi(\pi'(b')) = 0$, and hence,
  $$
  \phi(\pi'(b') \otimes b_{\mu+\nu}) = 0 = b_{\nu+w_\bullet\tau\nu} \otimes \pi''(b') \otimes b_\mu,
  $$
  as desired.
  Next, suppose that $\pi'(b') \neq 0$.
  Then, there exists $b_2 \in C_{I_\bullet;\mu+\nu}(b_{\lambda+\tau\nu})$ such that $\pi'(b') \otimes b_\mu \in C(b_2 \otimes b_{\mu+\nu})$.
  This implies that there exist $j_1,\ldots,j_r \in I_\bullet$ such that
  $\tilde{F}_{j_r} \cdots \tilde{F}_{j_1}(b_2 \otimes b_{\mu+\nu}) = \pi'(b') \otimes b_{\mu+\nu}$.
  Moreover, we must have
  \begin{align}
    \varepsilon_{j_{k+1}}(b_2^k) \geq \langle h_{j_{k+1}}, \mu+\nu \rangle,\ \varphi_{j_{k+1}}(b_2^k) > 0 \label{eq: vep and vphi for b_2^k}
  \end{align}
  for all $k = 0,\ldots,r-1$, where
  $$
  b_2^0 := b_2,\ b_2^{k+1} := \tilde{F}_{j_{k+1}} b_2^k.
  $$
  Then, we compute as
  \begin{align*}
    \phi(\pi'(b') \otimes b_{\mu+\nu}) &= \tilde{F}_{j_r} \cdots \tilde{F}_{j_1} \phi(b_2 \otimes b_{\mu+\nu}) \\
    &= \tilde{F}_{j_r} \cdots \tilde{F}_{j_1} (b_{\nu+w_\bullet\tau\nu} \otimes  \pi(b_2) \otimes b_\mu) \\
    &= b_{\nu+w_\bullet\tau\nu} \otimes \pi(b_2^r) \otimes b_\mu.
  \end{align*}
  The last equality follows from the tensor product rule \eqref{eq: tensor crystal} and equations \eqref{eq: iota_dot lm nu} and \eqref{eq: vep and vphi for b_2^k}.
  Since
  $$
  \pi(b_2^r) = \pi(\pi'(b')) = \pi''(b'),
  $$
  our claim follows.
  Thus, the proof completes.
\end{proof}

\begin{cor}\label{cor: pii with deltai assumption}
Let $\lambda,\mu,\nu \in X^+$.
Assume that there exists a based $\mathbf{U}^\imath$-module homomorphism $\delta^\imath_\nu: V(\nu+w_\bullet\tau\nu) \rightarrow \mathbb{Q}(q)$ such that $\delta^\imath_\nu(v_{\nu+w_\bullet\tau\nu}) = 1$.
Then, there exists a based $\mathbf{U}^\imath$-module homomorphism $\pi^\imath_{\lambda,\mu,\nu}: V(w_\bullet(\lambda+\tau\nu), \mu+\nu) \rightarrow V(w_\bullet\lambda, \mu)$ such that $\pi^\imath_{\lambda,\mu,\nu}(v_{w_\bullet(\lambda+\tau\nu)} \otimes v_{\mu+\nu}) = v_{w_\bullet\lambda} \otimes v_\mu$.
\end{cor}

\begin{proof}
  Let $f: V(w_\bullet(\lambda+\tau\nu), \mu+\nu) \rightarrow V(\nu+w_\bullet\tau\nu) \otimes V(w_\bullet\lambda, \mu)$ denote the based $\mathbf{U}$-module homomorphism in Proposition \ref{prop: existence of based U-hom}.
  Then, we have
  $$
  f(v_{w_\bullet(\lambda+\tau\nu)} \otimes v_{\mu+\nu}) = f(G(b_{w_\bullet(\lambda+\tau\nu)} \otimes b_{\mu+\nu})) = G(b_{\nu+w_\bullet\tau\nu} \otimes b_{w_\bullet\lambda} \otimes b_\mu) = v_{\nu+w_\bullet\tau\nu} \otimes v_{w_\bullet\lambda} \otimes v_\mu.
  $$

  Set $\pi^\imath = \pi^\imath_{\lambda,\mu,\nu} := (\delta^\imath_\nu \otimes \mathrm{id}) \circ f$.
  Noting that $G^\imath(b_{\nu+w_\bullet\tau\nu}) = v_{\nu+w_\bullet\tau\nu}$, we obtain
  $$
  \pi^\imath(v_{w_\bullet(\lambda+\tau\nu)} \otimes v_{\mu+\nu}) = (\delta^\imath_\nu \otimes \mathrm{id})(v_{\nu+w_\bullet\tau\nu} \otimes v_{w_\bullet\lambda} \otimes v_\mu) = v_{w_\bullet\lambda} \otimes v_\mu.
  $$
  This implies that $\pi^\imath$ commutes with the $\imath$bar-involutions.
  On the other hand, both $f$ and $\delta^\imath_\nu \otimes \mathrm{id}$ preserve the crystal lattices and $\mathcal{A}$-forms.
  Let $\phi: \mathcal{B}(w_\bullet(\lambda+\tau\nu), \mu+\nu) \rightarrow \mathcal{B}(\nu+w_\bullet\tau\nu) \otimes \mathcal{B}(w_\bullet\lambda, \mu) \sqcup \{0\}$ and $\delta^\imath_\nu: \mathcal{B}(\nu+w_\bullet\tau\nu) \rightarrow \{ 1 \} \sqcup \{0\}$ denote the induced maps.
  Then, $\pi^\imath$ preserves the crystal lattices, and the induced map
  $$
  \pi^\imath = \pi^\imath_{\lambda,\mu,\nu} := (\delta^\imath_\nu \otimes \mathrm{id}) \circ \phi: \mathcal{B}(w_\bullet(\lambda+\tau\nu),\mu+\nu) \rightarrow \mathcal{B}(w_\bullet\lambda,\mu) \sqcup \{0\}
  $$
  is injective on $\{ b \in \mathcal{B}(w_\bullet(\lambda+\tau\nu), \mu+\nu) \mid \pi^\imath(b) \neq 0 \}$.
  By Lemma \ref{lem: criterion of based ihom}, we see that $\pi^\imath$ is based.
  Thus, the proof completes.
\end{proof}

\section{Irreducible finite type of real rank $1$}\label{sec: main results}
In this section, we assume the following:
\begin{itemize}
  \item The admissible pair $(I, I_\bullet, \tau)$ is of irreducible finite type of real rank $1$ (cf. Subsection \ref{subsec: adm pair}).
  \item $Y = \sum_{i \in I} \mathbb{Z} h_i$.
  \item $X = \operatorname{Hom}_{\mathbb{Z}}(Y, \mathbb{Z})$.
  \item $\kappa_i = 0$ for all $i \in I_\circ$.
\end{itemize}

For each $i \in I$, let $\varpi_i \in X$ denote the $i$-th fundamental weight;
  $$
  \langle h_j, \varpi_i \rangle = \delta_{i,j}
  $$
for all $j \in I$.

In Subsections \ref{subsec: type A}--\ref{subsec: type F}, we study certain $\mathbf{U}^\imath$-modules in order to prove Proposition \ref{prop: base statement}.
Then, we prove the stability of $\imath$canonical bases in Theorem \ref{thm: main}.

\subsection{Type A}\label{subsec: type A}
Let $n \geq 1$ and consider the Dynkin diagram $I = \{ 1,\ldots,n \}$ of type $A_n$.
For each $r \in I$ and $n+1 \geq i_1 > \cdots i_r \geq 1$, let $b_{i_1,\ldots,i_r} \in \mathcal{B}(\varpi_r)$ denote the unique element of weight $\sum_{k=1}^r (\varpi_{i_k}-\varpi_{i_k-1})$; here, we set $\varpi_0 := 0$.
Set $v_{i_1,\ldots,i_r} := G(b_{i_1,\ldots,i_r}) \in V(\varpi_r)$.

\subsubsection*{Type \rm{AI}}
Set
\begin{align}
\begin{split}
  &\varsigma_1 := q^{-1}, \\
  &\varpi := 2\varpi_1.
\end{split} \label{eq: setting AI}
\end{align}
Since
$$
\varpi_1 + w_\bullet\tau\varpi_1 = 2\varpi_1,
$$
we have
$$
\nu+w_\bullet\tau\nu \in \mathbb{Z}_{\geq 0} \varpi
$$
for all $\nu \in X^+$.

\begin{lem}\label{lem: K type AI}
  Let $K: V(\varpi_1) \rightarrow V(\varpi_1)$ denote the linear isomorphism given by
  $$
  K(v_1) = v_2, \quad K(v_2) = v_1.
  $$
  Then, $K$ is a based $\mathbf{U}^\imath$-module isomorphism.
\end{lem}

\begin{proof}
  By Theorem \cite[Theorem 4.18]{BW18}, we see that there exists a $\mathbf{U}^\imath$-module isomorphism $K'$ such that $K'(v_1) = v_2$.
  Then, we have
  $$
  K'(v_2) = K'(B_1 v_1) = B_1 v_2 = v_1.
  $$
  Hence, we obtain $K' = K$.

  It remains to show that $K$ is based.
  We have
  $$
  v_2 = B_1 v_1 \in \mathcal{L}(\varpi_1) \cap {}_{\mathcal{A}} V(\varpi_1) \cap \psi^\imath(\mathcal{L}(\varpi_1 )),
  $$
  where $\psi^\imath$ denote the $\imath$bar-involution on $V(\varpi_1)$.
  This implies that $v_2 = G^\imath(b_2)$.
  Therefore, the $K$ is a based $\mathbf{U}^\imath$-module isomorphism.
  Thus, the proof completes.
\end{proof}

\begin{lem}\label{lem: w0 type AI}
There exists $w_0 \in \mathcal{L}(\varpi)$ such that $\mathbf{U}^\imath w_0 \simeq \mathbb{Q}(q)$ and $\mathrm{ev}_\infty(w_0) = b_\varpi$.
\end{lem}

\begin{proof}
  Noting that $\varpi = 2\varpi_1$ and $V(\varpi_1) = V(-\varpi_1)$ (see Remark \ref{rem: finite type}), consider the composition
  $$
  g: V(\varpi) \xrightarrow{\chi_{\varpi_1,\varpi_1}} V(\varpi_1) \otimes V(\varpi_1) \xrightarrow{K \otimes \mathrm{id}} V(\varpi_1) \otimes V(\varpi_1) \xrightarrow{\delta_{\varpi_1}} \mathbb{Q}(q).
  $$
  For the definitions of based $\mathbf{U}$-module homomorphisms $\chi_{\varpi_1, \varpi_1}$ and $\delta_{\varpi_1}$, see Example \ref{ex: based mods} \eqref{ex: based mods chi} and \eqref{ex: based mods delta}, respectively.

  By the definition, $g$ is a $\mathbf{U}^\imath$-module homomorphism preserving the crystal lattices.
  For each $b \in \mathcal{B}(\varpi)$, considering at $q = \infty$, we obtain
  $$
  \mathrm{ev}_\infty(g(G(b))) = \delta_{b, b_{\varpi}}.
  $$
  Hence, $\operatorname{Ker} g$ has a basis $\{ w_b \mid b \in \mathcal{B}(\varpi) \setminus \{ b_\varpi \} \}$ such that $w_b \in \mathcal{L}(\varpi)$ and $\mathrm{ev}_\infty(w_b) = b$.
  Therefore, the complement $W_0 \subset V(\varpi)$ of $\operatorname{Ker} g$ with respect to the inner product on $V(\varpi)$ (cf. Subsection \ref{subsec: V(pm lm)}) is spanned by a vector $w_0 \in \mathcal{L}(\varpi)$ satisfying $\mathrm{ev}_\infty(w_0) = b_\varpi$.
  Now, the assertion follows from the following:
  $$
  W_0 \simeq V(\varpi)/\operatorname{Ker} g \simeq \operatorname{Im} g = \mathbb{Q}(q).
  $$
  Thus, the proof completes.
\end{proof}

\subsubsection*{Type \rm{AII}}
Set
\begin{align}
\begin{split}
  &\varsigma_2 := q, \\
  &\varpi := \varpi_2.
\end{split} \label{eq: setting AII}
\end{align}
Since
$$
\varpi_i + w_\bullet \tau \varpi_i = \begin{cases}
    \varpi_2 & \text{ if } i = 1,3, \\
    2\varpi_2 & \text{ if } i = 2
  \end{cases}
$$
for all $i \in I$, we have
$$
\nu + w_\bullet\tau\nu \in \mathbb{Z}_{\geq 0} \varpi
$$
for all $\nu \in X^+$.

\begin{lem}\label{lem: calc AII}
We have $B_2 v_{4,3} = q^2 v_{3,1}$.
\end{lem}

\begin{proof}
  Using \cite[Proposition 5.2.2]{Lus10}, we compute as
  \begin{align*}
    &B_2 v_{4,3} = (F_2 + qT_{w_\bullet}(E_2)K_2^{-1})(v_{4,3}) \\
    &= q^2T_{w_\bullet}(E_2 T_{w_\bullet}^{-1}(v_{4,3})) \\
    &= q^2T_{w_\bullet}(E_2 v_{4,3}) \\
    &= q^2T_{w_\bullet}(v_{4,2}) \\
    &= q^2 v_{3,1}.
  \end{align*}
  This proves the assertion.
\end{proof}

\begin{lem}\label{lem: w0 type AII}
There exists $w_0 \in \mathcal{L}(\varpi)$ such that $\mathbf{U}^\imath w_0 \simeq \mathbb{Q}(q)$ and $\mathrm{ev}_\infty(w_0) = b_\varpi$.
\end{lem}

\begin{proof}
  By direct calculation and Lemma \ref{lem: calc AII}, we see that the vector
  $$
  w_0 := v_{2,1} - q^{-2}v_{4,3} \in V(\varpi)
  $$
  spans the $\mathbf{U}^\imath$-submodule isomorphic to $\mathbb{Q}(q)$.
  Clearly, we have $w_0 \in \mathcal{L}(\varpi)$ and $\mathrm{ev}_\infty(w_0) = b_{2,1} = b_\varpi$.
  Thus, the proof completes.
\end{proof}

\subsubsection*{Type \rm{AIII}}
Set
\begin{align}
\begin{split}
  &\varsigma_1 = \varsigma_2 := 1, \\
  &\varpi := \varpi_1 + \varpi_2.
\end{split} \label{eq: setting AIII}
\end{align}
Since
$$
\varpi_i + w_\bullet\tau\varpi_i = \varpi_1 + \varpi_2
$$
for all $i \in I$, we have
$$
\nu + w_\bullet\tau\nu \in \mathbb{Z}_{\geq 0} \varpi
$$
for all $\nu \in X^+$.

For each $i = 1,2$, the $V(\varpi_i)$ is two-dimensional.
Let $v_1^i$ (resp., $v_2^i$) denote the highest (resp., lowest) weight vector.

\begin{lem}\label{lem: K type AIII}
  Let $K: V(\varpi_1) \rightarrow V(\varpi_2)$ denote the linear isomorphism givne by
  $$
  K(v_1^1) = v_2^2, \ K(v_2^1) = v_1^2.
  $$
  Then, $K$ is a based $\mathbf{U}^\imath$-module isomorphism.
\end{lem}

\begin{proof}
  Noting that
  $$
  B_1 v^1_1 = v^1_2, \quad B_2 v^2_1 = v^2_2,
  $$
  the assertion can be proved by the same way as Lemma \ref{lem: K type AI}
\end{proof}

\begin{lem}\label{lem: w0 type AIII}
There exists $w_0 \in \mathcal{L}(\varpi)$ such that $\mathbf{U}^\imath w_0 \simeq \mathbb{Q}(q)$ and $\mathrm{ev}_\infty(w_0) = b_\varpi$.
\end{lem}

\begin{proof}
  Using Lemma \ref{lem: K type AIII}, the same argument as in the proof of Lemma \ref{lem: w0 type AI} proves the assertion.
\end{proof}

\subsubsection*{Type \rm{AIV}}
Set
\begin{align}
\begin{split}
  &\varsigma_1 := 1,\ \varsigma_n := (-1)^n q^{n-1} \\
  &\varpi := \varpi_1 + \varpi_n.
\end{split} \label{eq: setting AIV}
\end{align}
Since
$$
\varpi_i + w_\bullet\tau\varpi_i = \varpi_1 + \varpi_n
$$
for all $i \in I$, we have
$$
\nu + w_\bullet\tau\nu \in \mathbb{Z}_{\geq 0} \varpi
$$
for all $\nu \in X^+$.

\begin{lem}\label{lem: calc AIV}
  We have $B_n v_n = v_{n+1} + v_1$.
\end{lem}

\begin{proof}
  Using \cite[Proposition 5.2.2]{Lus10}, we compute as
  \begin{align*}
    &B_n v_n = (F_n + (-1)^nq^{n-1}T_{w_\bullet}(E_1)K_n^{-1})v_n \\
    &= v_{n+1} + (-1)^n q^{n-2} T_{w_\bullet}(E_1 T_{w_\bullet}^{-1}(v_n)) \\
    &= v_{n+1} + (-1)^2 T_{w_\bullet}(E_1 v_2) \\
    &= v_{n+1} + v_1.
  \end{align*}
  This proves the assertion.
\end{proof}

\begin{lem}\label{lem: K AIV}
  Let $K: V(\varpi_1) \rightarrow V(\varpi_1)$ denote the linear isomorphism givne by
  $$
  K(v_i) = \begin{cases}
    v_{n+1} & \text{ if } i = 1, \\
    v_i & \text{ if } i = 2,\ldots,n, \\
    v_1 & \text{ if } i = n+1.
  \end{cases}
  $$
  Then, it is a based $\mathbf{U}^\imath$-module isomorphism.
\end{lem}

\begin{proof}
  Noting that
  \begin{align*}
    &v_2 = B_1 v_1, \\
    &v_3 = B_2 v_2, \\
    &\cdots, \\
    &v_n = B_{n-1} v_{n-1}
  \end{align*}
  and that
  $$
  v_{n+1} = B_nv_n - v_1
  $$
  by Lemma \ref{lem: calc AIV}, the assertion can be proved by the same way as Lemma \ref{lem: K type AI}
\end{proof}

\begin{lem}\label{lem: w0 type AIV}
There exists $w_0 \in \mathcal{L}(\varpi)$ such that $\mathbf{U}^\imath w_0 \simeq \mathbb{Q}(q)$ and $\mathrm{ev}_\infty(w_0) = b_\varpi$.
\end{lem}

\begin{proof}
  Using Lemma \ref{lem: K AIV}, the same argument as in the proof of Lemma \ref{lem: w0 type AI} proves the assertion.
\end{proof}

\subsection{Type \rm{BII}}
Let $n \geq 2$ and consider the Dynkin diagram $I = \{ 1,\ldots,n \}$ of type $B_n$.
Set $L := \{ 1,\ldots,n,0,\bar{n},\ldots,\bar{1} \}$, and equip the set $\mathcal{B} := \{ b_i \mid i \in L \}$ with the crystal structure as in \cite[Example 2.22]{BS17}.
Then, $\mathcal{B}$ is identical to $\mathcal{B}(\varpi_1)$.
For each $i \in L$, set $v_i := G(b_i) \in V(\varpi_1)$.

Set
\begin{align}
\begin{split}
  &\varsigma_1 := q^{2n-3}, \\
  &\varpi := \varpi_1.
\end{split} \label{eq: setting BII}
\end{align}
Since
$$
\varpi_i + w_\bullet\tau\varpi_i = \begin{cases}
  2\varpi_1 & \text{ if } i \neq n, \\
  \varpi_1 & \text{ if } i = n
\end{cases}
$$
for all $i \in I$, we have
$$
\nu + w_\bullet\tau\nu \in \mathbb{Z}_{\geq 0} \varpi
$$
for all $\nu \in X^+$.

\begin{lem}\label{lem: calc BII}
  We have $B_1 v_{\bar{1}} = q^{2n-1} v_2$.
\end{lem}

\begin{proof}
  Using \cite[Proposition 5.2.2]{Lus10}, we compute as
  \begin{align*}
    &B_1 v_{\bar{1}} = (F_1 + q^{2n-3} T_{w_\bullet}(E_1)K_1^{-1}) v_{\bar{1}} \\
    &= q^{2n-1} T_{w_\bullet}(E_1 T_{w_\bullet}^{-1}(v_{\bar{1}})) \\
    &= q^{2n-1} T_{w_\bullet}(E_1 v_{\bar{1}}) \\
    &= q^{2n-1} T_{w_\bullet}(v_{\bar{2}}) \\
    &= q^{2n-1} v_2.
  \end{align*}
  This proves the assertion.
\end{proof}

\begin{lem}\label{lem: w0 type BII}
There exists $w_0 \in \mathcal{L}(\varpi)$ such that $\mathbf{U}^\imath w_0 \simeq \mathbb{Q}(q)$ and $\mathrm{ev}_\infty(w_0) = b_\varpi$.
\end{lem}

\begin{proof}
  By direct calculation and Lemma \ref{lem: calc BII}, we see that the vector
  $$
  w_0 := v_1 - q^{-2n+1}v_{\bar{1}} \in V(\varpi)
  $$
  spans the $\mathbf{U}^\imath$-submodule isomorphic to $\mathbb{Q}(q)$.
  Clearly, we have $w_0 \in \mathcal{L}(\varpi)$ and $\mathrm{ev}_\infty(w_0) = b_1 = b_\varpi$.
  Thus, the proof completes.
\end{proof}

\subsection{Type \rm{CII}}
Let $n \geq 3$ and consider the Dynkin diagram $I = \{ 1,\ldots,n \}$ of type $C_n$.
Set $L := \{ 1,\ldots,n, \bar{n}, \ldots, \bar{1} \}$, and equip the set $\mathcal{B} := \{ b_i \mid i \in L \}$ with the crystal structure as in \cite[Example 2.23]{BS17}.
Then, $\mathcal{B}$ is identical to $\mathcal{B}(\varpi_1)$.

The $\mathbf{U}$-module $V(\varpi_2)$ can be embedded into $V(\varpi_1) \otimes V(\varpi_1)$ in a way such that
$$
v_{\varpi_2} \mapsto v_2 \otimes v_1 - q^{-1} v_1 \otimes v_2.
$$
Then, the crystal basis $\mathcal{B}(\varpi_2)$ is identified with the connected component $C(b_2 \otimes b_1)$ of $\mathcal{B} \otimes \mathcal{B}$.
For each $(i,j) \in L^2$ such that $b_i \otimes b_j \in \mathcal{B}(\varpi_2)$, set
$$
  v_{i,j} = \begin{cases}
    v_{\bar{k}} \otimes v_k + q^{-1} v_{\overline{k-1}} \otimes v_{k-1} \\
    - q^{-1} v_{k-1} \otimes v_{\overline{k-1}} - q^{-2} v_k \otimes v_{\bar{k}} & \text{ if } (i,j) = (\overline{k}, k) \text{ for some } k \in \{2,\ldots,n\}, \\
    v_{i} \otimes v_{j} - q^{-1} v_{j} \otimes v_{i} & \text{ otherwise}.
  \end{cases}
$$
Then, $\{ v_{i,j} \mid b_i \otimes b_j \in \mathcal{B}(\varpi_2) \}$ forms a free basis of $\mathcal{L}(\varpi_2)$.

For each $i \in I$ and $k \in \{2,\ldots,n\}$, one can straightforwardly verify that
\begin{align}
  F_i v_{\bar{k},k} = \begin{cases}
    [2] v_{\bar{\imath}, i+1}, & \text{ if } k = i+1, \\
    v_{\bar{\imath}, i+1} & \text{ if } k = i+2 \text{ or } k = i < n, \\
    0 & \text{ otherwise},
  \end{cases} \label{eq: calc CII F}
\end{align}
and
\begin{align}  
  E_i v_{\bar{k},k} = \begin{cases}
    [2] v_{\overline{i+1}, i}, & \text{ if } k = i+1, \\
    v_{\overline{i+1}, i} & \text{ if } k = i+2 \text{ or } k = i < n, \\
    0 & \text{ otherwise}.
  \end{cases} \label{eq: calc CII E}
\end{align}

\begin{lem}\label{lem: calc CII 2}
Set
$$
w'_0 := -\frac{1}{[2]} v_{\bar{2}, 2} + \sum_{k=3}^n (-1)^{k-3} \frac{[n-k+1]}{[n-2]} v_{\bar{k}, k}.
$$
Then, we have $E_j w'_0 = F_j w'_0 = 0$ for all $j \in I_\bullet$.
\end{lem}

\begin{proof}
  Since $w'_0$ is a weight vector of weight $0$, the identity $E_j w'_0 = 0$ follows from $F_j w'_0 = 0$.

  By equation \eqref{eq: calc CII F}, we have
  $$
  F_i w'_0 = \begin{cases}
    [2] \cdot (-\frac{1}{[2]}) + 1 & \text{ if } i = 1, \\
    \frac{(-1)^{i-3}}{[n-2]}([n-i+1] - [2][n-i] + [n-i-1]) & \text{ if } i = 3,\ldots,n-1, \\
    0 & \text{ if } i = n.
  \end{cases}
  $$
  Now, the assertion follows from a well-known identity
  $$
  [2][m] = [m+1] + [m-1]
  $$
  valid for all $m \in \mathbb{Z}$.
  Thus, the proof completes.
\end{proof}

Set
\begin{align}
\begin{split}
  &\varsigma_2 := q^{n-1}, \\
  &\varpi := \varpi_2.
\end{split} \label{eq: setting CII}
\end{align}
Since
$$
\varpi_i + w_\bullet\tau\varpi_i = \begin{cases}
  \varpi_2 & \text{ if } i = 1 ,\\
  2\varpi_2 & \text{ if } i \neq 1
\end{cases}
$$
for all $i \in I$, we have
$$
\nu + w_\bullet\tau\nu \in \mathbb{Z}_{\geq 0} \varpi
$$
for all $\nu \in X^+$.

\begin{lem}\label{lem: w0 type CII}
There exists $w_0 \in \mathcal{L}(\varpi)$ such that $\mathbf{U}^\imath w_0 \simeq \mathbb{Q}(q)$ and $\mathrm{ev}_\infty(w_0) = b_\varpi$.
\end{lem}

\begin{proof}
  By direct calculation, equations \eqref{eq: calc CII F} and \eqref{eq: calc CII E}, and Lemma \ref{lem: calc CII 2}, we see that the vector
  $$
  w_0 := v_{2,1} - \frac{q^{-n+1}[2][n-2]}{[n]} w'_0 + q^{-2n+1} v_{\bar{1}, \bar{2}}
  $$
  spans the $\mathbf{U}^\imath$-submodule isomorphic to $\mathbb{Q}(q)$.
  Clearly, we have $w_0 \in \mathcal{L}(\varpi)$ and $\mathrm{ev}_\infty(w_0) = b_{2,1} = b_\varpi$.
  Thus, the proof completes.
\end{proof}

\subsection{Type \rm{DII}}
Let $n \geq 4$ and consider the Dynkin diagram $I = \{ 1,\ldots,n \}$ of type $D_n$.
Set $L := \{ 1,\ldots,n,\bar{n},\ldots,\bar{1} \}$, and equip the set $\mathcal{B} := \{ b_i \mid i \in L \}$ with the crystal structure as in \cite[Example 2.24]{BS17}.
Then, $\mathcal{B}$ is identical to $\mathcal{B}(\varpi_1)$.
For each $i \in L$, set $v_i := G(b_i) \in V(\varpi_1)$.

Set
\begin{align}
\begin{split}
  &\varsigma_1 := q^{n-2}, \\
  &\varpi := \varpi_1.
\end{split} \label{eq: setting DII}
\end{align}
Since
$$
\varpi_i + w_\bullet\tau\varpi_i = \begin{cases}
  2\varpi_1 & \text{ if } i \neq n-1,n ,\\
  \varpi_1 & \text{ if } i = n-1,n
\end{cases}
$$
for all $i \in I$, we have
$$
\nu + w_\bullet\tau\nu \in \mathbb{Z}_{\geq 0} \varpi
$$
for all $\nu \in X^+$.

\begin{lem}\label{lem: calc DII}
  We have $B_1 v_{\bar{1}} = q^{n-1} v_2$.
\end{lem}

\begin{proof}
  Using \cite[Proposition 5.2.2]{Lus10}, we compute as
  \begin{align*}
    &B_1 v_{\bar{1}} = (F_1 + q^{n-2} T_{w_\bullet}(E_1)K_1^{-1}) v_{\bar{1}} \\
    &= q^{n-1} T_{w_\bullet}(E_1 T_{w_\bullet}^{-1}(v_{\bar{1}})) \\
    &= q^{n-1} T_{w_\bullet}(E_1 v_{\bar{1}}) \\
    &= q^{n-1} T_{w_\bullet}(v_{\bar{2}}) \\
    &= q^{n-1} v_2.
  \end{align*}
  This proves the assertion.
\end{proof}

\begin{lem}\label{lem: w0 type DII}
There exists $w_0 \in \mathcal{L}(\varpi)$ such that $\mathbf{U}^\imath w_0 \simeq \mathbb{Q}(q)$ and $\mathrm{ev}_\infty(w_0) = b_\varpi$.
\end{lem}

\begin{proof}
  By direct calculation and Lemma \ref{lem: calc DII}, we see that the vector
  $$
  w_0 := v_1 - q^{-n+1}v_{\bar{1}} \in V(\varpi)
  $$
  spans the $\mathbf{U}^\imath$-submodule isomorphic to $\mathbb{Q}(q)$.
  Clearly, we have $w_0 \in \mathcal{L}(\varpi)$ and $\mathrm{ev}_\infty(w_0) = b_1 = b_\varpi$.
  Thus, the proof completes.
\end{proof}

\subsection{Type \rm{FII}}\label{subsec: type F}
Consider the Dynkin diagram $I = \{ 1,2,3,4 \}$ of type $F_4$, and the crystal $\mathcal{B}(\varpi_4)$ (see \cite[Fig. 5.8]{BS17}).
For each $\lambda \in X \setminus \{ 0 \}$ with $\mathcal{B}(\varpi_4)_\lambda \neq \emptyset$, let $b_\lambda \in \mathcal{B}(\varpi_4)$ denote the unique element of weight $\lambda$.
Set
$$
b_0^1 := \tilde{F}_4 b_{-\varpi_3+2\varpi_4}, \quad b_0^2 := \tilde{F}_3 b_{-\varpi_2+2\varpi_3-\varpi_4}.
$$
Then, we have $\mathcal{B}(\varpi_4)_0 = \{ b_0^1, b_0^2 \}$.

Set
\begin{align}
\begin{split}
  &\varsigma_1 := q^5, \\
  &\varpi := \varpi_4.
\end{split} \label{eq: setting FII}
\end{align}
Since
$$
\varpi_i + w_\bullet\tau\varpi_i = \begin{cases}
  2\varpi_4 & \text{ if } i = 1,4, \\
  4\varpi_4 & \text{ if } i = 2, \\
  3\varpi_4 & \text{ if } i = 3
\end{cases}
$$
for all $i \in I$, we have
$$
\nu + w_\bullet\tau\nu \in \mathbb{Z}_{\geq 0} \varpi
$$
for all $\nu \in X^+$.

\begin{lem}\label{lem: w0 type FII}
There exists $w_0 \in \mathcal{L}(\varpi)$ such that $\mathbf{U}^\imath w_0 \simeq \mathbb{Q}(q)$ and $\mathrm{ev}_\infty(w_0) = b_\varpi$.
\end{lem}

\begin{proof}
  Using a computer program GAP \cite{GAP} with a package QuaGroup \cite{GT22}, we see that
  $$
  w_0 := G(b_{\varpi_4}) - \frac{q^{-5}[2]}{[3]}\left(G(b_0^2) - \frac{1}{[2]} G(b_0^1)\right) + q^{-11} G(b_{-\varpi_4})
  $$
  spans a $\mathbf{U}^\imath$-submodule of $V(\varpi)$ isomorphic to $\mathbb{Q}(q)$.
  Clearly, we have $w_0 \in \mathcal{L}(\varpi)$ and $\mathrm{ev}_\infty(w_0) = b_{\varpi_4} = b_\varpi$.
  Thus, the proof completes.
\end{proof}

\subsection{Stability of $\imath$canonical bases}
In the sequel, we set $\varsigma_i$ for $i \in I_\circ$ and $\varpi \in X^+$ as in \eqref{eq: setting AI}--\eqref{eq: setting BII}, \eqref{eq: setting CII}--\eqref{eq: setting FII}.

\begin{prop}\label{prop: base statement}
  For each $m \in \mathbb{Z}_{\geq 0}$, there exists a unique based $\mathbf{U}^\imath$-module homomorphism $g_m: V(m\varpi) \rightarrow \mathbb{Q}(q)$ such that $g_m(v_{m\varpi}) = 1$.
\end{prop}

\begin{proof}
  We prove the assertion by induction on $m$.
  The case when $m = 0$ is trivial.
  Let us prove for $m = 1$.

  Let $w_0 \in V(\varpi)$ be as in Lemmas \ref{lem: w0 type AI}, \ref{lem: w0 type AII}, \ref{lem: w0 type AIII}, \ref{lem: w0 type AIV},, \ref{lem: w0 type BII}, \ref{lem: w0 type CII}, \ref{lem: w0 type DII}, and \ref{lem: w0 type FII}.
  Set $W_0 := \mathbb{Q}(q) w_0$ and $W_1 \subset V(\varpi)$ the complement of $W_0$ with respect to the inner product in Subsection \ref{subsec: V(pm lm)}.
  Then, we can write
  $$
  v_\varpi = cw_0 + w_1
  $$
  for some $c \in 1 + \mathbf{A}_\infty$ and $w_1 \in W_1$.
  Define a $\mathbf{U}^\imath$-module homomorphism
  $$
  g_1: V(\varpi) = W_0 \oplus W_1 \rightarrow \mathbb{Q}(q)
  $$
  by
  $$
  g_1(w_0) = c^{-1},\ g_1(W_1) = 0.
  $$
  Then, it preserves the crystal lattices, and we have
  $$
  g_1(v_\varpi) = 1.
  $$
  This identity implies that $g_1$ preserves the $\mathcal{A}$-forms and commutes with the $\imath$bar-involutions.
  Moreover, the induced map $\gamma_1: \mathcal{B}(\varpi) \rightarrow \{1\} \sqcup \{0\}$ satisfies
  $$
  \gamma_1(b) = \delta_{b, b_\varpi}
  $$
  for all $b \in \mathcal{B}(\varpi)$.
  Therefore, by Lemma \ref{lem: criterion of based ihom}, we see that $g_1$ is a based $\mathbf{U}^\imath$-module homomorphism.
  
  Now, assume that $m > 1$ and the assertion is true for $m-1$.
  Consider the composition
  $$
  g: V(m\varpi) \xrightarrow{\chi_{(m-1)\varpi, \varpi}} V((m-1)\varpi) \otimes V(\varpi) \xrightarrow{g_{m-1} \otimes \mathrm{id}} V(\varpi) \xrightarrow{g_1} \rightarrow \mathbb{Q}(q).
  $$
  Then, it preserves the crystal lattices and $\mathcal{A}$-forms, and we have
  $$
  g(v_{m\varpi}) = g_1(g_{m-1}(v_{(m-1)\varpi}) \otimes v_{\varpi}) = g_1(v_\varpi) = 1.
  $$
  This identity implies that $g = g_m$ and that $g_m$ commutes with the $\imath$bar-involutions.
  Moreover, the induced map $\gamma_m: \mathcal{B}(m\varpi) \rightarrow \{1\} \sqcup \{0\}$ satisfies
  $$
  \gamma(b) = \delta_{b, b_{m\varpi}}
  $$
  for all $b \in \mathcal{B}(m\varpi)$.
  Therefore, by Lemma \ref{lem: criterion of based ihom}, we see that $g_m$ is a based $\mathbf{U}^\imath$-module homomorphism.
  Thus, the proof completes.
\end{proof}

\begin{thm}\label{thm: main}
Let $\lambda,\mu,\nu \in X^+$.
Then, there exists a unique based $\mathbf{U}^\imath$-module homomorphism
$$
\pi^\imath_{\lambda,\mu,\nu}: V(w_\bullet(\lambda+\tau\nu), \mu+\nu) \rightarrow V(w_\bullet\lambda,\mu)
$$
such that
$$
\pi^\imath_{\lambda,\mu,\nu}(v_{w_\bullet(\lambda+\tau\nu)} \otimes v_{\mu+\nu}) = v_{w_\bullet\lambda} \otimes v_\mu.
$$
\end{thm}

\begin{proof}
  The assertion follows from Corollary \ref{cor: pii with deltai assumption} and Proposition \ref{prop: base statement} since we have $\nu + w_\bullet\tau\nu \in \mathbb{Z}_{\geq 0} \varpi$.
\end{proof}

\begin{cor}
  For each $\lambda,\mu \in X^+$, there exists a unique based $\mathbf{U}^\imath$-module homomorphism
  $$
  \pi^\imath_{\lambda,\mu}: \dot{\mathbf{U}}^\imath 1_{\overline{\mu+w_\bullet\lambda}} \rightarrow V(w_\bullet\lambda,\mu)
  $$
  such that
  $$
  \pi^\imath_{\lambda,\mu}(1_{\overline{\mu+w_\bullet\lambda}}) = v_{w_\bullet\lambda} \otimes v_\mu.
  $$
\end{cor}

\begin{proof}
  Define a $\mathbf{U}^\imath$-module homomorphism $\pi^\imath_{\lambda,\mu}$ by
  $$
  \pi^\imath_{\lambda,\mu}(x) := x \cdot (v_{w_\bullet\lambda} \otimes v_\mu).
  $$
  We shall show that it is based.

  Let $x \in \dot{\mathbf{B}}^\imath$ (see Example \ref{ex: based Ui-mod and hom} \eqref{ex: based Ui-mod and hom, Uidot}).
  By \cite[Theorem 7.2]{BW21}, there exists $\nu \in X^+$ such that
  $$
  \pi^\imath_{\lambda+\tau\nu,\mu+\nu}(x) \in \mathbf{B}^\imath(w_\bullet(\lambda+\tau\nu), \mu+\nu).
  $$
  Then, by Theorem \ref{thm: main}, we see that
  $$
  \pi^\imath_{\lambda,\mu}(x) = \pi^\imath_{\lambda,\mu,\nu}(\pi^\imath_{\lambda+\tau\nu,\mu+\nu}(x)) \in \mathbf{B}^\imath(w_\bullet\lambda, \mu) \sqcup \{0\}.
  $$
  
  It remains to show that $\operatorname{Ker} \pi^\imath_{\lambda,\mu}$ is spanned by a subset of $\dot{\mathbf{B}}^\imath$.
  Let $v \in \operatorname{Ker} \pi^\imath_{\lambda,\mu}$, and write
  $$
  v = \sum_{k=1}^r c_k x_k
  $$
  for some $c_1,\ldots,c_r \in \mathbb{Q}(q)^\times$ and $x_1,\ldots,x_r \in \dot{\mathbf{B}}^\imath$.
  Again by \cite[Theorem 7.2]{BW21}, there exists $\nu \in X^+$ such that the vectors $\pi^\imath_{\lambda+\tau\nu, \mu+\nu}(x_k)$ are distinct $\imath$canonical basis elements.
  Since the homomorphism $\pi^\imath_{\lambda,\mu,\nu}$ is based, the nonzero vectors of the form $\pi^\imath_{\lambda,\mu}(x_k) = \pi^\imath_{\lambda,\mu,\nu}(\pi^\imath_{\lambda+\tau\nu, \mu+\nu}(x_k))$ are distinct.
  Hence, we must have $\pi^\imath_{\lambda,\mu}(x_k) = 0$ for all $k = 1,\ldots,r$.
  This implies that
  $$
  \operatorname{Ker} \pi^\imath_{\lambda,\mu} = \mathbb{Q}(q)\{ x \in \dot{\mathbf{B}}^\imath \mid \pi^\imath_{\lambda,\mu}(x) = 0 \},
  $$
  as desired.
  Thus, the proof completes.
\end{proof}

\end{document}